\documentclass[a4paper,11pt]{article}
\usepackage[plainpages=false]{hyperref}
\usepackage{amsfonts,latexsym,rawfonts,amsmath,amssymb,amsthm, mathrsfs, lscape, calligra}
\usepackage{verbatim}
\usepackage{enumerate}
\usepackage{centernot}
\usepackage{mathtools}
\usepackage{tikz}
\usepackage{tikz-cd} 
\usepackage{authblk}

\usepackage[all]{xy}
\usepackage{graphicx,psfrag}

\usepackage{array, tabularx, color}

\usepackage{setspace}

\newtheorem{thm}{Theorem}[section]
\newtheorem{cor}[thm]{Corollary}
\newtheorem{lem}[thm]{Lemma}
\newtheorem{clm}[thm]{Claim}
\newtheorem{prop}[thm]{Proposition}

\newtheorem{notation}[thm]{Notation}

\theoremstyle{remark}
\newtheorem{rmk}[thm]{Remark}

\theoremstyle{definition}
\newtheorem{defi}[thm]{Definition}
\def \Gr {Gr}
\def \m {\mathfrak m}
\def \R {\mathbb{R}}
\def \Z {\mathbb{Z}}
\def \G {\mathcal{G}}
\def \C {\mathbb{C}}
\def \E {\mathcal{E}}
\def \F {\mathcal{F}}
\def \O {\mathcal{O}}
\def \N {\mathcal N}
\def \P {\mathbb{CP}}

\def \Tr {\text{Tr}}

\DeclareMathOperator \Sing {Sing}
\DeclareMathOperator\dVol {dVol}
\DeclareMathOperator\Vol {Vol}
\DeclareMathOperator\rank {rank}
\DeclareMathOperator\Ker {Ker}
\DeclareMathOperator \Id {Id}

\numberwithin{equation}{section}
\begin{document}

\newcommand{\bC}{{\textbf C}}
\newcommand{\Quot}{{\bf Quot}}

\title{Reflexive sheaves, Hermitian-Yang-Mills connections, and tangent cones}
\author{Xuemiao Chen\thanks{University of Maryland; xmchen@umd.edu.}, Song Sun\thanks{University of California, Berkeley; sosun@berkeley.edu. }}
\date{\today}
\maketitle{}

\begin{abstract}
 In this paper we give a complete algebro-geometric characterization of analytic tangent cones of admissible Hermitian-Yang-Mills connections over any reflexive sheaves. 
\end{abstract}

\tableofcontents

\section{Introduction} 

Let $\omega$ be a smooth K\"ahler metric on the ball $B=\{|z|<1\}$ in $\C^n$, and let $\E$ be a reflexive coherent sheaf defined on a neighborhood of $\overline B$. Let  $A$ be an admissible Hermitian-Yang-Mills (HYM) connection on $\E$ with respect to the K\"ahler metric $\omega$. This means that $A$ is the Chern connection of a smooth Hermitian-Yang-Mills (HYM) metric $H$ on the locally free locus of $\E$, and the curvature  $F_A$ has finite $L^2$ norm on any compact subset of $B$. Denote by $\Sing(\E)$ the singular locus of $\E$. It is known that $\Sing(\E)$ is a complex analytic subset of complex co-dimension at least 3. Our main goal in this paper is to understand the singular behavior of $A$ at $0$ in terms of the underlying sheaf $\E$.

We first briefly review the notions of analytic and algebraic tangent cones, and  more details will be provided in Section \ref{section2}. The notion of   \emph{analytic tangent cones} is first studied by Tian \cite{Tian}.  Consider the dilation map $$\lambda: B_{\lambda^{-1}}\rightarrow B, z\mapsto \lambda \cdot z$$ where $\lambda>0$, and the family of rescaled connections $$A_\lambda=\lambda^*A.$$ Letting $\lambda \rightarrow 0$, passing to a subsequence and  applying gauge transforms, $A_\lambda$ converges smoothly to a  connection $A_\infty$ on $\C^n\setminus (\Sigma\cup Z(\E))$, and $A_\infty$ is HYM with respect to the standard flat metric. Here $\Sigma$  is the \emph{bubbling set}, i.e.,  the subset of $\C^n\setminus Z(\E)$ where the convergence is not smooth, and $Z(\E)$ is the Zariski tangent cone of $\Sing(\E)$ at $0$. It is known by \cite{BS, Tian} that $A_\infty$ extends to an admissible HYM connection on a reflexive sheaf $\E_\infty$ on $\C^n$, and $\Sigma\cup Z(\E)$ is a closed affine algebraic subvariety in $\C^n$.  Moreover, passing to a further subsequence, there is a complex codimension 2 algebraic cycle on $\P^{n-1}$, called the \emph{analytic blow-up cycle}, which is of the form 
$$\Sigma_b^{an}=\sum_{k} m_k^{an}\cdot [\underline\Sigma_b^k],$$
such that the affine cone over the support $\cup_k \underline\Sigma_b^k$ of $\Sigma_b^{an}$ is precisely the pure complex codimension  two part of $\Sigma\cup Z(\E)$, and the \emph{analytic multiplicity} $m_k^{an}$ is a positive integer characterizing the blow-up of  Yang-Mills energy transverse to $\Sigma_b^k$ at a generic point. On the other hand, the rest of $\Sigma$ is contained in $\Sing(A_\infty)=\Sing(\E_\infty)$.   We call the pair $(A_\infty, \Sigma_b^{an})$ an analytic tangent cone. 

The terminology ``tangent cone" is justified by the fact that $A_\infty$ is a HYM cone connection in the sense of \cite{CS1} (see Section \ref{section2.1}). The underlying sheaf $\E_\infty$ is of the form $\psi_*\pi^*\underline \E_\infty$, where $$\pi: \C^n\setminus \{0\}\rightarrow \P^{n-1}$$ is the natural projection map and $$\psi: \C^n\setminus \{0\}\rightarrow \C^n$$ is the inclusion map, and $\underline\E_\infty$ is a reflexive sheaf on $\P^{n-1}$ which is a direct sum of polystable\footnote{Throughout this paper, when we talk about stability of sheaves on the projective space, we always mean \emph{slope} stability with respect to the standard polarization.} sheaves. Moreover,  up to gauge equivalence $A_\infty$ is uniquely determined by the sheaf $\underline\E_\infty$. So the information of an analytic tangent cone is completely encoded in the algebraic data $(\underline\E_\infty, \underline\Sigma_b^{an})$.   We point out that a priori from the definition analytic tangent cones at $0$ depend on not only the initial connection $A$, but also on the choice of subsequences as $\lambda\rightarrow 0$. Uniqueness of tangent cones independent of subsequences is in general a difficult question in many geometric analytic problems.

Recall from \cite{CS1, CS3}  we introduced the notion of an \emph{algebraic tangent cone} at a singularity of a reflexive coherent analytic sheaf $\E$. This is defined to be a torsion-free sheaf $\underline{\widehat\E}$ on the exceptional divisor  $D=\P^{n-1}$ that is given by the restriction  of a reflexive extension $\widehat\E$ of $p^*(\E|_{B\setminus\{0\}})$ across $D$, where $$p: \widehat B\rightarrow B$$ is the blowup at $0$. In general algebraic tangent cones  at $0$ are not necessarily unique, due to the fact that the exceptional divisor has complex co-dimension exactly 1. We say an algebraic tangent cone $\underline{\widehat \E}$ is \emph{optimal} if 
$$\Phi(\underline{\widehat\E}):=\mu(\underline{\E}_1)-\mu(\underline \E_m /\underline \E_{m-1})\in [0, 1),$$ where $$0=\underline{\E}_0\subset \underline{\E}_1\subset \cdots \subset\underline{\E}_m=\underline{\widehat\E}$$ is the  Harder-Narasimhan filtration of $\underline{\widehat\E}$, and $\mu(\cdot)$ denotes the slope of a torsion-free sheaf on $\P^{n-1}$ with respect to the standard polarization. The function $\Phi$ measures how far an algebraic tangent cone is from being semistable.

Given a  torsion-free sheaf $\underline\E$ on $\P^{n-1}$, we  denote by $\Gr^{HN}(\underline{\E})$ (resp. $\Gr^{HNS}(\underline\E))$ the graded sheaf associated to the Harder-Narasimhan (resp. Harder-Narasimhan-Seshadri) filtration of $\underline{\E}$. 
In \cite{CS3} it is proved that an optimal algebraic tangent cone always exists and it is unique up to certain natural transforms. In particular, the isomorphism class of the corresponding graded torsion-free sheaf $\Gr^{HN}(\underline{\widehat\E})$, up to tensoring each factor by some $\O(k)$, does not depend on the choice of optimal algebraic tangent cones. For our purpose, we need to consider instead $Gr^{HNS}(\underline{\widehat \E})$. The latter is not unique in general but certain algebraic invariants can be extracted. More specifically, we define a reflexive sheaf over $\C^n$ 
\begin{equation}\label{e:definition of Galg}
\mathcal{G}^{alg}:=\psi_*\pi^*(Gr^{HNS}(\underline{\widehat \E}))^{**},
\end{equation}
and a complex codimension 2 algebraic cycle  $$\Sigma_b^{alg}:=
\Sigma_b^{alg}(\underline{\widehat \E})$$  on $\P^{n-1}$ (c.f. Definition \ref{defi2.13}). We call $\Sigma_b^{alg}$ the \emph{algebraic blow-up cycle} of $\E$ at $0$. It is a fact that  both $\mathcal G^{alg}$ and $\Sigma_b^{alg}$ do not depend on the specific choice of optimal algebraic tangent cones, so they are local algebraic invariants of the stalk of $\E$ at $0$ (see Section \ref{section2.2} for more details).

In \cite{CS3}  we made a conjecture relating the analytic and algebraic tangent cones, motivated by the results in \cite{CS1, CS2}. In this paper we give a proof of this conjecture in complete generality, based on  the techniques introduced in \cite{CS1, CS2} and a new approach. Simply put, the algebraic data underlying the analytic tangent cones matches exactly with the above algebraic invariants of optimal algebraic tangent cones. More precisely, we have

\begin{thm}\label{main}
Given $\E$ and $A$ as  above, then
there is a unique analytic tangent cone $(A_\infty,\Sigma_b^{an})$ at $0$, which is completely determined by the stalk of $\E$ at $0$:
\begin{itemize}
\item[(I).]  $A_\infty$ is gauge equivalent to the HYM cone on $\G^{alg}$ (see Section \ref{section2.1} for the definition of a HYM cone). In particular, $\E_\infty$ is isomorphic to $\G^{alg}$, and
$$\Sing(A_\infty)=\Sing(\G^{alg}).$$ 
\item[(II)] The analytic blow-up cycle $\Sigma_b^{an}$ equals the algebraic blow-up cycle $\Sigma_b^{alg}$.  
\end{itemize}
\end{thm}

In \cite{CS1, CS2} we proved this result under the extra assumption that $0$ is an isolated and homogeneous singularity of $\E$, or that there is an algebraic tangent cone $\underline{\widehat{\E}}$ which is locally free and stable (see also \cite{JWE} for a special case). The arguments there are more analytical and involve PDE estimates based on explicit construction of good background Hermitian metrics. In the general setting the previous approach meets severe difficulties. In this paper, we proceed using a different idea which is more intrinsic and is of more algebraic nature.
The key new input in this paper is that, as we will show in Section \ref{section3}, an admissible HYM connection $A$ naturally recovers all the equivalence classes of  optimal algebraic tangent cones (see Section \ref{section3}.)

In Section \ref{section2} we include some background material. In Section \ref{section4} we finish the proof of Theorem \ref{main}, using recent results   on moduli compactification of admissible HYM connections  and semi-stable sheaves over projective manifolds (see \cite{GSRW} and \cite{GT}).

\

\textbf{Remark on notations}:  In this paper we need to introduce many notations regarding sheaves on different spaces. The convention is that a coherent sheaf $\E$ is denoted by a letter with calligraphic font; a coherent sheaf $\underline\E$ on the projective space $\P^{n-1}$ comes with an extra  underline; a coherent sheaf $\widehat \E$ on the blowup $\widehat B$ of the ball $B$ comes with an extra widehat.

\

\noindent \textbf{Acknowledgements}:  Both authors were partially supported by the Simons Collaboration Grant on Special Holonomy in Geometry, Analysis and Physics (No. 488633, S.S.), and the second author was partially supported by NSF grant  DMS-1916520. The first author would like to thank Santai Qu for helpful discussions. He also thanks UC Berkeley for hospitality during his visit between January 2018 and May 2019.  The second author thanks Princeton University for hospitality during Fall 2019. We also thank the anonymous referees for valuable comments which helped improve the exposition of the paper.

\section{Preliminaries}
\label{section2}

\subsection{Analytic tangent cones}\label{section2.1}
In this subsection, we recall some backgrounds about analytic tangent cones of Hermitian-Yang-Mills connections. For details see \cite{CS1}. As in the introduction  we let $\E$ be a reflexive sheaf defined over a neighborhood of $\overline{B}$  in $\C^n$, where $\overline{B}=\{z\in \C^n: |z|\leq 1\}$. We assume that under the standard holomorphic coordinates $\{z_1, \cdots, z_n\}$ on $\C^n$ 
\begin{equation}\label{Kahler metric expansion}\omega=\omega_0 + O(|z|^2)  ,
\end{equation}
where 
$$\omega_0:=\frac{\sqrt{-1}}{2} \partial \bar{\partial} |z|^2$$
is the standard flat metric. Let $H$ be an admissible HYM metric on $\E$. This means that $H$ is a smooth Hermitian metric on $\E$ outside $\Sing(\E)$, and the associated Chern connection $A$ satisfies the following conditions
\begin{itemize}
\item the HYM equation
 \begin{equation}\label{HYMequation}\sqrt{-1} \Lambda_\omega F_A =c\cdot \Id	
 \end{equation}
 holds on $B\setminus \Sing(\E)$, for some constant $c\in \R$, called the \emph{Einstein constant};
\item $A$ has locally finite Yang-Mills energy:
$$\int_{K\setminus\Sing(\E)} |F_A|^2\dVol_\omega < \infty$$ for any compact subset $K \subset  B$.
\end{itemize}
Notice that a Hermitian-Yang-Mills connection is a projectively unitary $\Omega$-anti-self-dual instanton in the sense of Tian \cite{Tian}, so is \emph{stationary} by \cite{Tian}, Proposition 5.1.2. Hence by  Proposition 5.1.1 in \cite{Tian},   Price's monotonicity formula holds. In particular, we have

\begin{equation}\label{eqn2.0}
\sup_{r\in (0,1]}r^{4-2n}\int_{B_r\setminus \Sing(\E)} |F_A|^2\dVol_{\omega} <\infty.
\end{equation}
For any $\lambda>0,$ we denote the rescaling map $$\lambda: B_{\lambda^{-1}} \rightarrow B; z \mapsto \lambda \cdot z,$$ where in this paper $B_r$ always denotes the ball $\{|z|<r\}$ in $\C^n$. For any sequence of positive numbers $\lambda_j \rightarrow 0$, we get a sequence of admissible HYM connections $$A_j:=\lambda_j^* A$$ with respect to the K\"ahler metric $\omega_j:=\lambda_j^{-2}\cdot \lambda_j^*\omega$. By \eqref{eqn2.0}, this sequence has uniformly bounded Yang-Mills energy over any compact subset $K \subset \C^n$. Notice by \eqref{Kahler metric expansion} obviously $\omega_j$ converges smoothly to $\omega_0$ as $j\rightarrow\infty$. 
 
We denote by $Z(\E)$ the $\C^*$ invariant reduced subvariety in $\C^n$ underlying the Zariski tangent cone of $\Sing(\E)$, i.e., $Z(\E)$ is the Hausdorff limit of $\Sing(\E)$ in $\C^n$ under the above rescaling as $\lambda\rightarrow 0$. It is well-known that $Z(\E)$ has the same complex dimension as $\Sing(\E)$ (see for example \cite{HW}), so is of complex codimension at least 3 in $\C^n$.

The associated \emph{bubbling set} of this sequence is defined as
\begin{equation}\label{bs}
\Sigma=\{z\in \C^n\setminus Z(\E)| \lim_{r\rightarrow 0} \liminf_{j\rightarrow\infty} r^{4-2n}\int_{B_z(r)} |F_{A_j}|^2\dVol_{\omega_j}\geq \epsilon_0\},
\end{equation}
where $\epsilon_0$ denote the $\epsilon$-regularity constant for Yang-Mills connections over the flat $\C^n$ (see Theorem $2.2.1$ in \cite{Tian} for example).
Applying Uhlenbeck's $\epsilon$-regularity theorem and standard analytic results on the convergence of Yang-Mills connections, by passing to a subsequence we may assume there is a smooth  connection $A_\infty$ on some Hermitian vector bundle $(\E_\infty, H_\infty)$ defined over $\C^n\setminus (\Sigma\cup Z(\E))$, which is HYM with respect to the flat metric $\omega_0$. Moreover, for $j\gg1$, there exist Hermitian isomorphisms $$P_j: (\lambda_j^*\E, \lambda_j^*H)\rightarrow (\E_\infty, H_\infty)$$ such that $({P_j^{-1}})^*(A_j)$ converges smoothly to $A_\infty$\footnote{In this paper, when we talk about convergence of a sequence of objects, we often need to pass to a subsequence. We will abuse notation and not re-lable the new subsequence, if no confusion arises.}.   The connection $A_\infty$ can be viewed as an admissible HYM connection on $\C^n$, hence by Bando-Siu we know \cite{BS} $\E_\infty$ extends to a reflexive sheaf on $\C^n$ and $A_\infty$ extends smoothly outside $\Sing(\E_\infty)\subset \Sigma\cup Z(\E)$, so that the set of essential singularities of $A_\infty$ is given by $\Sing(A_\infty)=\Sing(\E_\infty)$.

\

It is  proved by Tian (see Theorem $4.3.3$ in \cite{Tian}) that $\Sigma$ is a complex-analytic set in $\C^n\setminus Z(\E)$, and the complement $\Sigma\setminus \Sing(A_\infty)$ has pure complex codimension 2. Since $Z(\E)$ is of complex codimension at least 3, by the Remmert-Stein-Shiffman extension theorem (see \cite{Shiffman} for example) we know the closure $\Sigma_b$ of the pure complex codimension 2 part of $\Sigma$ in $\C^n$ is also a pure codimension 2 complex analytic set in $\C^n$.  Let $\{\Sigma_b^k\}$ be the irreducible components of $\Sigma_b$. Then by Theorem 4.3.3 in \cite{Tian},  passing to a subsequence we may assume the convergence of Radon measures on $\C^n\setminus Z(\E)$  
 \begin{equation}\label{energy identity}\frac{1}{8\pi^2}|F_{A_j}|^2 \dVol_{\omega_j}\rightharpoonup \frac{1}{8\pi^2} |F_{A_\infty}|^2\dVol_{\omega_0}+\nu, 
 \end{equation}
with $$\nu=\sum_{k} m_{k}^{an}\cdot \mathcal H^{2n-4}\llcorner (\Sigma_b^k\setminus Z(\E)), $$
where $\mathcal H^{2n-4}$ denotes the $2n-4$ dimensional Hausdorff measure, and  $m_k^{an}$ are positive integers called the analytic multiplicities. 
Now by definition we have
$$(\cup_k\Sigma_b^k )\cup\Sing(A_\infty)\subset\Sigma\cup Z(\E).$$
Although not needed in this paper, one can show that the two sets above are indeed identical, using Lemma 3.18 in \cite{GSRW}.
 
Fixing an irreducible component $\Sigma_b^k$, and taking a generic complex 2-dimensional  slice $\Delta$ which intersects $\Sigma_b^k$ transversely, we have the following formula computing analytic multiplicities
\begin{lem}[\cite{SW}, Lemma $4.1$]\label{analyticmultiplictyformula}
$$m_k^{an}=\lim_{j\rightarrow\infty} \frac{1}{8\pi^2}\int_{\Delta}\text{Tr}(F_{A_{j}} \wedge F_{A_{j}})-\Tr(F_{A_\infty}\wedge F_{A_\infty}).$$
\end{lem}

In our special setting, one can say more about the structure of the limiting data (See \cite{Tian, CS1, CS2})
\begin{enumerate}[(1).]
\item $\Sigma$ is invariant under the natural action of $\C^*$ on $\C^n$. 
\item $ F_{A_\infty}(\partial_r, \cdot)=0$, where $\partial_r$ is the radial vector field on $\C^n$;
\item $\sqrt{-1}\Lambda_{\omega_0}F_{A_\infty}=0$, away from $\Sing(\E_\infty)$.
\end{enumerate}
In particular, each $\Sigma_b^k$ is an affine cone over a pure codimension 2 algebraic subvariety $\underline \Sigma_b^k$ in $\P^{n-1}$. 

\begin{defi}
The \emph{analytic blow-up cycle} is the codimension 2 algebraic cycle on $\P^{n-1}$ given by 
$$\Sigma_b^{an}:=\sum_{k}m_k^{an}\cdot[\underline\Sigma_b^k]$$ 
\end{defi}

\begin{defi}
We call the pair $(A_\infty, \Sigma_{b}^{an})$ an \emph{analytic tangent cone} of $A$ (or $(\E, H))$ at $0$. 
\end{defi}
\begin{rmk}
The definition here is different from \cite{CS2}, where an analytic tangent cone is defined to be the triple $(A_\infty, \Sigma, \nu)$. But it is easy to see that these two definitions contain exactly the same information. The above definition is more convenient to use in this paper.  
\end{rmk}

\begin{rmk} \label{rmk:sequence}
We emphasize here that the analytic tangent cones are \emph{not} a priori unique, and may depend on the choice of subsequences as $\lambda\rightarrow 0$. But it is not difficult to see that any analytic tangent cone can be obtained by taking the rescaled limit corresponding to a subsequence of the fixed sequence $\{\lambda_j:=2^{-j}\}$. We shall use this fact in Section \ref{section3}.
\end{rmk}

The above extra properties in our setting imply that  $A_\infty$ is a HYM cone in the following sense. 
 Let $\underline \F$ be a polystable  reflexive sheaf over $\P^{n-1}$ with slope $\mu$. By Theorem $3$ in \cite{BS}, there exists an admissible HYM metric $\underline H$ on $\underline \F$ with respect to the Fubini-Study metric. Now on the reflexive sheaf $\F=\psi_*\pi^*\underline \F$, the metric $|z|^{2\mu}\pi^*\underline H$ is again an admissible HYM metric, with respect to the flat metric $\omega_0$, and with vanishing Einstein constant. We let $A_{\F}$ be the associated Chern connection which is an admissible HYM connection. Such $(\F, A_{\F})$ is called a \emph{simple HYM cone}. It should be noted here that tensoring $\underline \F$ with some $\O(k)$ does not change the resulting simple HYM cone. 

\begin{defi} \label{definition2.5}
A direct sum of simple HYM cones is called a HYM cone. In particular, it is determined by a direct sum of polystable reflexive sheaves on $\P^{n-1}$. 
\end{defi}

The above properties (2) and (3) implies that  
\begin{lem}[\cite{CS1}, Theorem $2.23$]\label{lem2.6}
The tangent cone connection $A_\infty$ is a HYM cone on $\E_\infty$. More precisely, we can write $\E_\infty=\psi_*\pi^*\underline\E_\infty$, where 
\begin{equation}\label{decomposition}\underline\E_\infty=\bigoplus _i\underline \E_{\infty}^{\mu_i}
\end{equation}
 so that each $\underline \E_{\infty}^{\mu_i}$ is a polystable reflexive sheaf with slope given by $\mu_i\in [0,1)$, with $\mu_i \neq \mu_j$ if $i\neq j$, and $A_\infty$ is gauge equivalent to $\bigoplus_i A_i$ where $A_i$ is the simple HYM cone determined by $\underline \E_{\infty}^{\mu_i}$.
\end{lem}
\begin{rmk}
It follows from the discussion in \cite{CS1} that given $A_\infty$, such $\underline\E_\infty$  and the above decomposition are unique up to isomorphism, under the normalization condition that $\mu_i\in [0, 1)$. 
\end{rmk}

To sum up,  an analytic tangent cone $(A_\infty, \Sigma_b^{an})$ is  uniquely determined by the corresponding algebraic data $\E_\infty$ and $\Sigma_b^{an}$ up to isomorphisms. 
Our main result Theorem \ref{main} thus gives an algebro-geometric characterization of $\E_\infty$ and $\Sigma_b^{an}$ in terms of $\E$ itself.

\

 For our purpose later, we also need to discuss the notion of convergence of holomorphic sections. Given a sequence of holomorphic sections $\sigma_j$ of $$\E_j:=\lambda_j^*\E$$ with uniformly bounded $L^2$ norm over $B$ and a holomorphic section $\sigma_\infty$ on some analytic tangent cone $\E_\infty$ over $B$, then 
 
\begin{defi}
 We say  $\sigma_j$ converges to $\sigma_\infty$ if  $P_j(\sigma_j)$ converges to $\sigma_\infty$ smoothly outside $\Sigma\cup Z(\E)$. 
\end{defi}    

The following is essentially a consequence of \cite{BS}. For the convenience of readers we include a sketch of proof here. 

\begin{lem}
For any compact subset $K\subset B$, there is a constant $C=C(K)$ independent of $j$ such that 
\begin{equation}\label{BS interior estimate}
|\sigma_j|^2_{L^\infty(K)}\leq C \int_{B} |\sigma_j|^2\dVol_{\omega_j}
 \end{equation}
 \end{lem}

\begin{rmk}
 Here and later in this paper, the norm of a holomorphic section is always meant to be the one defined by the natural Hermitian metric, and the integral is always taken on the complement of the singular set of the sheaf. 
\end{rmk}
 
\begin{proof} From Theorem 2 (b) in \cite{BS} we know $|\sigma_j|^2$ is locally bounded in $B$ and is smooth away from $\Sing(\E_j)$. By the HYM equation one computes $$\Delta |\sigma_j|^2=2|\nabla\sigma_j|^2 -2c\cdot \lambda_j^2|\sigma_j|^2\geq -2c\cdot \lambda_j^2|\sigma_j|^2$$ on $B\setminus \Sing(\E_j)$, where $c$ is the Einstein constant of the original admissible HYM metric on $\E$ (c.f. \eqref{HYMequation}). By the fact that $\Sing(\E_j)$ has complex codimension at least 3 and $|\sigma_j|^2$ is locally bounded, using a cut-off function and integration by parts  one then sees that $|\sigma_j|^2\in W^{1,2}_{loc}(B)$, and the inequality $\Delta|\sigma_j|\geq -c\cdot \lambda_j^2|\sigma_j|$ holds in the weak sense on $B$. Then the conclusion follows from local Moser iteration.
 \end{proof}

 It follows easily from \eqref{BS interior estimate} and the Hartogs's extension theorem (see \cite{Shiffman}) for holomorphic sections of reflexive sheaves, that given a sequence $\{\sigma_j\}$ with $\int_B |\sigma_j|^2$ uniformly bounded, one can always extract a  convergent subsequence. Moreover  since the set $\Sigma\cup Z(\E)$ has vanishing Lebesgue measure,  we have
\begin{equation}\label{lower semi continuity}\int_B |\sigma_\infty|^2\dVol_{\omega_0}\leq \liminf_{j\rightarrow\infty} \int_B |\sigma_j|^2\dVol_{\omega_j}.
\end{equation}

We may also refer to the above convergence as \emph{weak} convergence.
\begin{defi}
We say $\sigma_j$ \emph{strongly} converges to $\sigma_\infty$ if $\{\sigma_j\}$ converges to $\sigma_\infty$, and furthermore $$\int_B |\sigma_\infty|^2\dVol_{\omega_0}= \lim_{j\rightarrow\infty} \int_B |\sigma_j|^2\dVol_{\omega_j}.$$ 
\end{defi}
Again since $\Sigma\cup Z(\E)$ has vanishing Lebesgue measure, it is clear that strong convergence follows from convergence if one can establish an a priori bound $$\int_{B} |\sigma_j|^{{2+\epsilon}}\dVol_{\omega_j}\leq C$$ for some $\epsilon>0$. In reality we will indeed derive a uniform $L^\infty$ bound to guarantee strong convergence, see \cite{CS1} and Section \ref{Finiteness of degree}. In view of \eqref{BS interior estimate}, the key point is to rule out the blowing up of $L^\infty$ norm near $\partial B$. The following fact will be used in Section \ref{A torsion-free sheaf on $D$}.  

\begin{lem}
\label{facts about strong convergence}
Suppose $\sigma_j$ and $\sigma_j'$ converge strongly to $\sigma_\infty$ and $\sigma_\infty'$ respectively, and $f$ is a fixed holomorphic function on $B$, then 
\begin{itemize}
\item $\sigma_j+\sigma_j'$ converges strongly to $\sigma_\infty+\sigma_\infty'$;
\item $f\cdot\sigma_j$ converges strongly to $f\cdot\sigma_\infty$.
\end{itemize}
\end{lem}
\begin{proof}
It suffices to notice that if $\sigma_j$ converges to $\sigma_\infty$, then the convergence is strong if and only if 
$$\lim_{r\rightarrow 1^-}\limsup_{j\rightarrow\infty}\int_{B\setminus B_{r}}|\sigma_j|^2\dVol_{\omega_j}=0.$$
 This is itself a consequence of \eqref{BS interior estimate} and \eqref{lower semi continuity}.
\end{proof}

\subsection{Algebraic tangent cones}\label{section2.2}
In this subsection, we collect the results on algebraic tangent cones of reflexive sheaves. For details see \cite{CS3}. 
We fix a reflexive  sheaf $\E$ over $B \subset \C^n$. Let $$p: \widehat{B} \rightarrow B$$ denote the blowup of $B$ at  $0\in B$ and denote by $D=\P^{n-1}$ the exceptional divisor. Then we define $\mathcal{A}$ to be the set of isomorphism classes of reflexive sheaves $\widehat\E$ over $\widehat{B}$ so that  $\widehat{\E}|_{\widehat{B}\setminus D}$ is isomorphic to $p^*\E|_{\widehat{B} \setminus 0}$. An element $\widehat{\E} \in \mathcal{A}$ is called an \emph{extension} of $\E$ at $0\in B$ and the torsion-free sheaf $$\underline{\widehat \E}: =\iota_D^*\widehat\E$$ is called an algebraic tangent cone of $\E$ at $0$, where 
$$\iota_D: D\rightarrow \widehat B$$
denotes the obvious inclusion map. We define a function 
$$\Phi: \mathcal{A} \rightarrow \mathbb{Q}_{\geq 0}; \ \ \ \ \widehat\E\mapsto \mu(\underline \E_1)-\mu(\underline \E_m/\underline \E_{m-1}),$$ where $$0= \underline\E_0\subset \underline \E_1 \subset \cdots \underline \E_m=\underline{\widehat \E}$$ is the Harder-Narasihman filtration of $\underline{\E}$. 
\begin{defi}
$\widehat{\E}\in \mathcal{A}$ is called an optimal extension of $\E$ at $0$ if $\Phi(\widehat{\E})\in [0,1).$
In this case we also call $\underline{\widehat \E}$ an optimal algebraic tangent cone of $\E$ at $0$. 
\end{defi}
Notice $\widehat\E$ is optimal if $\underline{\widehat\E}$ is semi-stable, i.e., $\Phi(\underline{\widehat{\E}})=0$. But simple examples (see \cite{CS3}) show that we can not always achieve semistability and this is the reason for introducing the weaker notion of being optimal. 

Now given  an optimal extension $\widehat{\E}$, and a subsheaf $\underline \E_i$ occurring in the Harder-Narasimhan filtration of $\E$,  we define the Hecke transform (or elementary transform)  of $\widehat{\E}$ along $\underline \E_i$ to be the reflexive sheaf $\widehat{\E}^i$ which is given by the following natural exact sequence 
$$
0\rightarrow \widehat{\E}^i \rightarrow \widehat{\E} \rightarrow (\iota_D)_*(\underline{\widehat{\E}}/\underline \E_i) \rightarrow 0
$$
Then by Corollary $3.3$ in \cite{CS3}, we know that $\widehat{\E}^i$ is again an optimal extension of $\E$ at $0$. We also say $\widehat{\E}^i$ and $\widehat{\E}$ differ by a \emph{Hecke transform of special type}. It is also shown that the graded sheaves $\Gr^{HN}(\underline{\widehat\E})$ and $\Gr^{HN}(\underline{\widehat\E^i})$ are isomorphic up to tensoring each factor by some $\O(k)$ on $D$.

Next it is easy to see that if $\widehat\E\in \mathcal A$, then for any $k\in \Z$, the sheaf $\widehat \E':=\widehat\E\otimes [D]^k$ is again an extension,  where $[D]$ denotes the line bundle on $\widehat B$ defined by the the divisor $D$. In this case we say $\widehat \E'$ and $\widehat \E$ are \emph{equivalent} extensions. Restricting to $D$, we have $\underline{\widehat \E'}=\underline{\widehat \E}\otimes \O(-k)$, so in particular $\widehat \E$ is optimal if and only if $\widehat\E'$ is optimal. 

\begin{thm}[\cite{CS3}]\label{optimalalgebraictangentcone}
Given a reflexive sheaf $\E$ over $B$, we have 
\begin{itemize}
\item An optimal extension of $\E$ at $0\in B$ always exists and up to equivalence, two  optimal extensions differ by a Hecke transform of special type. In particular, there are exactly $m$ different optimal extensions up to equivalence, where $m$ is the length of the Harder-Narasimhan filtration of $\underline{\widehat\E}$. 
\item The graded sheaf $\Gr^{HN}(\underline {\widehat{\E}})$  is uniquely determined by $\E$ up to tensoring each factor with some $\mathcal O(k)$. In particular the sheaf $\psi_*\pi^*(\Gr^{HN}(\underline{\widehat \E}))$ on $\C^n$ is uniquely determined by $\E$. 
\end{itemize}
\end{thm}

\begin{rmk}
If we want  strict \emph{uniqueness} of optimal extensions, we can impose the normalizing condition that $\mu(\underline\E_1)$ and $\mu(\underline\E_m/\underline\E_{m-1})$ are both in the interval $[0, 1)$. This will remove the freedom of performing Hecke transform of special type or tensoring with $\O(k)$. For our purpose in this paper, the statement in Theorem \ref{optimalalgebraictangentcone} is more suitable since it implies that each factor of  $\Gr^{HN}(\underline{\widehat \E})$ can be viewed as the maximal destabilizing subsheaf of \emph{some} optimal algebraic tangent cone. 
\end{rmk}

For our purpose, we need to consider the sheaf $\Gr^{HNS}(\underline{\widehat{\E}})$ which in general depends on the choice of the Harder-Narasihman-Seshadri filtration of $\underline{\widehat \E}$ (see Example $3.1$ in \cite{BTT}). Nonetheless, we can still extract algebraic invariants from $\Gr^{HNS}(\underline{\widehat{\E}})$ that suffice for the need in this paper.

\

 We first introduce a general definition following \cite{BTT}.

\begin{defi} \label{defi210}
Given a torsion sheaf $\mathcal T$ on $\P^{n-1}$ with support in codimension at least 2, we define
 the codimension 2 \emph{support} cycle of $\mathcal T$ to be the algebraic cycle
$$\mathcal C(\mathcal T):=\sum_k m_k^{alg}\cdot [\underline\Sigma_k]$$ 
where $\underline\Sigma_k$ are  irreducible codimension $2$ components of the support of $\mathcal T$, and  the \emph{algebraic multiplicity} $$m_k^{alg}=h^0(\underline \Delta, \mathcal T|_{\underline \Delta})$$ for a  complex $2$ dimensional slice $\underline\Delta$ which intersects $\underline\Sigma_k$ transversely at a generic point. 
\end{defi}

\begin{defi}\label{defi2.13}
Let $\underline \F$ be a  torsion-free coherent sheaf on $\P^{n-1}$.
We define   the pure codimension $2$ algebraic cycle $\Sigma_b^{alg}(\underline\F)$ of $\underline \F$ to be 
$$\Sigma_b^{alg}(\underline \F):=\mathcal C(\mathcal T), $$
where 
$\mathcal T=(Gr^{HNS}(\underline \F))^{**}/Gr^{HNS}(\underline \F)$.
\end{defi}
\begin{rmk}
It follows from the definition that $\Sigma^{alg}_b(\underline \F)$ is supported on the pure codimension $2$ part of the support of $\mathcal{T}$.  
When $\underline \F$ is locally free,  we know by Proposition $2.3$ in \cite{SW} that away from $\Sing((Gr^{HNS}(\underline \F))^{**})$, the support of $\mathcal{T}$  has pure codimension $2$. 
\end{rmk}

We have
 
 \begin{prop}[\cite{BTT} Proposition 2.1]
 Given a semistable torsion free sheaf $\underline\F$ on $\P^{n-1}$, the reflexive sheaf $(\Gr^{HNS}(\underline{\F}))^{**}$ and the codimension 2 algebraic cycle $\Sigma_b^{alg}(\underline\F)$ do not depend on the choice of the Harder-Narasihman-Seshadri filtration of $\underline\F$.
 \end{prop}

An immediate corollary is 

\begin{cor}\label{cor2-18}
For a reflexive sheaf $\E$ over $B$, let $\widehat{\E}$ be any optimal extension of $\E$ at $0\in B$, then the sheaf 
\begin{equation}
\mathcal G^{alg}:=\psi_*\pi^*(\Gr^{HNS}(\underline{\widehat\E}))^{**}
\end{equation}
and the algebraic blow-up cycle
\begin{equation}\Sigma_b^{alg}:=\Sigma_b^{alg}(\underline{\widehat\E})
\end{equation}
 do not depend on the choice of the optimal extension $\widehat\E$ at $0$, hence are algebraic invariants of the stalk of $\E$ at $0$. 
\end{cor}
 
 \begin{rmk}Since $(\Gr^{HNS}(\underline{\widehat\E}))^{**}$ is a direct sum of polystable sheaves, we can apply the construction in Section \ref{section2.1} to obtain a  HYM cone on the sheaf $\G^{alg}$. By Corollary \ref{cor2-18}, we know that up to gauge equivalence such a HYM cone is \emph{canonically} associated to the stalk of $\E$ at $0$. 
 \end{rmk}
\subsection{Moduli of semi-stable sheaves}\label{section2.3}
In this subsection, we will review some algebro-geometric results from \cite{GSRW, GT}. We mention that these results will not be used until  Section \ref{uniquenessoflimitingsheaf}. The results are proved on general polarized projective manifolds but for our purpose we will only consider the case when the base manifold is $\P^{n-1}$ with the standard polarization $\O(1)$.

Let $\underline\E$ be a semi-stable torsion-free sheaf on $\P^{n-1}$.  Denote by $r$ the rank of $\underline\E$, and  denote by $\tau$ the Hilbert polynomial of $\underline \E$. 
Throughout this paper we shall denote $$\underline \E(k):=\underline \E\otimes \O(k).$$
Since the set of semi-stable torsion-free  sheaves having the same Hilbert polynomials as $\underline \E$ forms a bounded family (see \cite{Maruyama}), we  may fix $k$ large  so that for any such $\underline\E'$, $\underline \E'(k)$ is globally generated, and  $$H^i(\P^{n-1}, \underline \E'(k))=0$$ for all $i>0$. 
Denote the sheaf $$\mathcal{H}={\C^{\oplus \tau(k)}} \otimes \O(-k).$$ Choosing a basis of $H^0(\P^{n-1}, \underline \E(k))$ gives an exact sequence  
$$
\mathcal{H} \xrightarrow{\underline q} \underline \E \rightarrow 0, 
$$
hence yields a point in the Quot scheme $\Quot(\mathcal{H}, \tau)$. Here $\Quot(\mathcal{H}, \tau)$ is by definition the set of equivalence classes of quotients $$q: \mathcal{H} \rightarrow \underline \E' \rightarrow 0,$$ where $\underline \E'$ is a coherent sheaf over $\P^{n-1}$ with Hilbert polynomial equal to $\tau$. Two quotients $\underline q_1: \mathcal{H} \rightarrow \underline \E_1$ and $\underline q_2: \mathcal{H} \rightarrow \underline \E_2$ are equivalent if $\Ker(\underline q_1)=\Ker(\underline q_2)$; this is the same as saying that there exists an isomorphism $\underline \rho: \underline \E_1 \rightarrow \underline \E_2$ so that $$\underline{\rho} \circ \underline q_1=\underline q_2.$$ Notice there is a natural action of $GL(\tau(k), \C)$ on $\Quot(\mathcal{H}, \tau)$ given by $$M.\underline p'=  \underline p' \circ M$$ for any $M\in GL(\tau(k), \C)$ and any quotient $\underline p'$ in $\Quot(\mathcal{H},\tau)$.

By \cite{Grothendieck}, we know $\Quot(\mathcal{H}, \tau)$ is a projective scheme which admits a decomposition $$\Quot(\mathcal{H}, \tau)= \coprod \Quot(\mathcal{H}, (c_1,\cdots, c_{\min(r,n-1)})),$$ where $\Quot(\mathcal{H}, (c_1,\cdots, c_{\min(r,n-1)}))$ consists of those quotients with fixed Chern classes $(c_1,\cdots, c_{\min(r,n-1)})$.  Let $$\Quot(\mathcal{H}, c(\underline\E))=\Quot(\mathcal{H}, c_1(\underline \E), \cdots c_{\min(r, n-1)}(\underline \E)).$$ 
Now we denote by $R^{\mu ss}\subset \Quot(\mathcal{H}, c(E))$  the subscheme of quotients $q:\mathcal{H} \rightarrow \underline \E' \rightarrow 0$ so that 
\begin{itemize}
\item $\underline\E'$ is torsion-free;
\item $\det(\underline \E')\cong \det(\underline \E)$; 
\item $\underline \E'$ is semi-stable;
\item $q$ induces an isomorphism ${\C}^{\oplus \tau(k)} \cong H^0(\P^{n-1}, \underline \E'(k))$.
\end{itemize}
Let $\mathcal{Z}$ denote the reduced weak normalization of $R^{\mu ss}$ as a complex analytic space. This means that we first take the underlying reduced complex analytic space of $R^{\mu ss}$, and then take its weak normalization, so that every locally defined \emph{continuous} function on $\mathcal{Z}$ which is holomorphic on the smooth part $\mathcal{Z}_{reg}$ of $\mathcal{Z}$ is in fact holomorphic. Notice the weak normal property is ``weaker" than being normal, in the sense that we impose the extension property only for continuous functions. As a simple example, in complex dimension 1, a nodal singularity is weakly normal but not normal,  and a cusp singularity is not weakly normal.

\begin{prop}[\cite{GT},  Definition $4.4.$ and Theorem $5.5$]\label{prop2.4}
There exists a compact complex  analytic space $M^{\mu ss}$, with a natural continuous map $$\phi: \mathcal{Z} \rightarrow M^{\mu ss}$$  so that
\begin{itemize}
\item The image of a fixed $GL(\tau(k), \C)$ orbit in $\mathcal{Z}$ is a point;
\item If two quotients $\underline q_i: \mathcal{H} \rightarrow \underline \E_i, i=1,2$ in $\mathcal{Z}$ have the same image in $M^{\mu ss}$, then 
$$
(\Gr^{HNS}(\underline \E_1))^{**}\cong (\Gr^{HNS}(\underline \E_2))^{**},$$
and
$$ \Sigma_b^{alg}(\underline \E_1)=\Sigma_b^{alg}(\underline \E_2).
$$
\end{itemize}
\end{prop}

For our purpose in this paper, we also need the following fact regarding the convergence of a sequence in the space $\Quot(\mathcal{H}, \tau)$ in the analytic topology. We fix a smooth Hermitian metric on $\mathcal H$. Given any sequence of $M_i \in GL(\tau(k), \C)$, we define $$q_i=q\circ M_i: \mathcal{H} \rightarrow \underline \E \rightarrow 0.$$ Furthermore, we assume $q_i$ converge to $$q_\infty: \mathcal{H} \rightarrow \underline \E_\infty$$  in the analytic topology of $\Quot(\mathcal{H},\tau)$.  

The maps $q_i$ (resp. $q_\infty$) induce smooth bundle endomorphisms $\pi_i: \mathcal H \rightarrow \mathcal H$ (resp. $\pi_\infty: \mathcal H \rightarrow \mathcal H$), which are given by  projection onto the orthogonal complement of $\Ker(q_i)$ (resp. $\Ker(q_\infty)$ in $\mathcal H$, and are defined away from $\Sing(\underline\E)$ (resp. $\Sing(\underline\E_\infty)$). We have

\begin{lem}[\cite{GSRW}]\label{quoteschemeanalyticconvergence}
Away from $\Sing(\underline \E) \cup \Sing(\underline \E_\infty)$, $\pi_i$ converges to $\pi_\infty$ smoothly.
\end{lem}

This follows exactly the same as the proof of Lemma 2.17 in \cite{GSRW}. For the convenience of readers we reproduce the arguments here.  The key point is a geometric interpretation of the abstract convergence of $q_i$ to $q_\infty$ in $\Quot$. To see this we use the fact that for some fixed $m$ large,  the Quot scheme $\Quot$ admits an embedding into a fixed Grassmannian $Gr(W, r)$ where $W=H^0(\P^{n-1}, \mathcal H (m))$ and $r=\dim W-\tau(m)$. Now by fixing a Hermitian metric on $W$, $q_i$ (resp. $q_\infty$) can be viewed as a vector space endomorphism $P_{q_i}$ (resp. $P_{q_\infty}$) of $W$ given by orthogonal projection onto the subspace $H^0(\P^{n-1}, \Ker(q_i)(m))$ (resp. $H^0(\P^{n-1}, \Ker(q_\infty)(m))$. The convergence of $q_i$ to $q_\infty$ implies that  $P_{q_i}$ converges to $P_{q_\infty}$. On the other hand,  we can also view $P_{q_i}$ (resp. $P_{q_\infty}$) as a bundle endomorphism of the holomorphic vector bundle corresponding to the sheaf $W\otimes \O_{\P^{n-1}}$, and  the image of the natural composition 
$$
W\otimes \O_{\P^{n-1}} \xrightarrow{P_{q_i} (\text{resp}. \ P_{q_\infty}) } W\otimes \O_{\P^{n-1}} 
\rightarrow \mathcal{H}(m)
$$
is exactly given by $\Ker(q_i)(m)$ (resp. $\Ker(q_\infty)(m)$) away from $\text{Sing}(\underline \E)$ (resp. $\text{Sing}(\underline \E_\infty)$). This yields that away from $\Sing(\underline \E)\cup \Sing(\underline \E_\infty)$, the sub-bundles $\Ker(q_i)$ converge smoothly to $\Ker(q_\infty)$. This clearly implies  the convergence of $\pi_i$.

Combining with the continuity of the map $\phi$ in Proposition \ref{prop2.4} we obtain
\begin{cor}\label{S-equivalence}
If $\underline \E_\infty$ is a torsion-free and  semi-stable, then $$(\Gr^{HNS}(\underline \E_\infty))^{**}=(\Gr^{HNS}(\underline \E))^{**}$$ and $$\Sigma^{alg}_b(\underline \E_\infty^{alg})=\Sigma^{alg}_b(\underline \E).$$ 
Furthermore, if $\underline \E_\infty^{**}$ is polystable, then 
$$\Sigma^{alg}_b(\underline\E)=\mathcal C(\underline\E_\infty^{**}/\underline\E_\infty).$$
\end{cor}
\begin{proof}
The first part follows from Proposition \ref{prop2.4} directly since $q_i$ stays in a fixed $GL(\tau(k), \C)$ orbit. It remains to prove the second part. Indeed, we take a Seshadri filtration of $\underline \E_\infty$ as 
$$
0\subset \underline\E_1 \subset \cdots \underline\E_m=\underline \E_\infty 
$$
where $\underline\E_i$ are saturated in $\underline\E_\infty$. Since $\underline \E_\infty^{**}$ is polystable, we have a canonical isomorphism $\underline{\E}_\infty^{**}/\underline{\E_1}^{**} \cong (\underline{\E}_\infty/\underline \E_1)^{**}$. Consequently, using the fact that $\underline\E_\infty/\underline\E_1$ is torsion-free, we obtain the following exact sequence 
$$
0\rightarrow (\underline \E_1)^{**}/\underline \E_1 \rightarrow \underline{\E}_\infty^{**}/\underline{\E}_\infty \rightarrow (\underline{\E}_\infty/ \underline \E_1)^{**}/ (\underline{\E}_\infty/ \underline \E_1) \rightarrow 0.
$$
Since each term has support in codimension at least 2, it follows from Lemma 2.15 in \cite{GSRW} that 
$$\mathcal C(\underline{\E}_\infty^{**}/\underline{\E}_\infty)=\mathcal C((\underline{\E}_1)^{**}/\underline{\E}_1)+\mathcal C((\underline{\E}_\infty/ \underline \E_1)^{**}/ (\underline{\E}_\infty/ \underline \E_1)).$$
Since by assumption $\underline\E_\infty/\underline \E_1$ is again a torsion-free sheaf whose double dual is polystable, the conclusion follows from induction by repeating the process for $\underline \E_\infty / \underline \E_1$.
\end{proof}

\section{Optimal algebraic tangent cones from admissible HYM}
\label{section3}
The goal of this subsection is to show that an admissible HYM connection naturally  gives rise to optimal algebraic tangent cones. More precisely,  in Section \ref{Finiteness of degree} we study general properties of the degree function introduced in \cite{CS1}; in Section \ref{A torsion-free sheaf on $D$} we define certain canonical torsion-free sheaves on the exceptional divisor $D$ of the blowup $p: \widehat B\rightarrow B$; in Section \ref{section3.3} we show these torsion-free sheaves do arise as algebraic tangent cones; in Section \ref{Optimality} we show these algebraic tangent cones are optimal.

\subsection{Properties of the degree function}
\label{Finiteness of degree}
We first recall the definition of the degree function in \cite{CS1}. Given $(\E, H, A)$ as in the introduction, we denote by $\E_0$ the stalk of $\E$ at $0$. Then the degree function
\begin{equation}
d: \E_0\rightarrow \mathcal \R\cup \{\infty\}
\end{equation}
is defined 
by setting  that $d(s)=\infty$ if $s=0$ is the zero section, and that for a non-zero holomorphic section $s$ defined in a neighborhood of $0$,
\begin{equation}
d(s):= \lim_{r\rightarrow 0^+}\frac{\log \int_{B_r} |s|^2\dVol_{\omega}}{2\log r}-n.
\end{equation}
\begin{lem}
$d(s)$ is  well-defined and lies in  $(\rank (\E)!)^{-1}\mathbb Z \cup \{\infty\}$
\end{lem}
\begin{proof}
This is proved in Corollary 3.7 in \cite{CS1} in the setting when $0$ is an isolated singularity, but this assumption is not essentially used there.  Two key points are the interior estimate \eqref{BS interior estimate}, and the Hartogs's extension theorem.
\end{proof}

The same definition applies to an analytic tangent cone $\E_\infty$. Because of  the cone structure we have a notion of \emph{homogeneous holomorphic sections} on $\E_\infty$. We say a holomorphic section $s$ of $\E_\infty$ is homogeneous of degree $\beta$ if away from $\Sing(\E_\infty)$, we have 
$$\nabla_{\partial_r}s=\beta r^{-1}s. 
$$
It is easy to see for such $s$, $d(s)=\beta$ using the above definition of the degree function. The following Lemma will be used in the next section. 

\begin{lem}\label{lem-finite cone}
Given any analytic tangent cone $\E_\infty$, then for any fixed $\beta\in \mathbb R$,  the space $V_\beta$ of homogeneous holomorphic sections on $\E_\infty$ of degree $\beta$ is finite dimensional. 
\end{lem}
\begin{proof}
We define a norm on $V_{\beta}$  by setting
$$\|s\|_{L^2(B)}^2:=\int_{B} |s|^2\dVol_{\omega_0}. $$
It suffices to show that the unit sphere in $V_\beta$ is compact. Given a sequence $s_j\in V_\beta$ with $\|s_j\|_{L^2(B)}=1$, then after passing to a subsequence we obtain a weak limit $s_\infty$ with $\|s_\infty\|_{L^2(B)}\leq 1$. We need to show the equality holds.  Notice  that away from $\Sing(\E_\infty)$, $s_j$ converges smoothly to $s_\infty$. So it suffices to prevent mass concentration near $\partial B$.  The key point is that  by homogeneity we have for all $j$
\begin{equation}
\int_{B_2} |s_j|^{2}\dVol_{\omega_0}=2^{2n+2\beta} 
\end{equation}
So by \eqref{BS interior estimate} we obtain $\|s_j\|_{L^\infty(B)}\leq C$ for a uniform $C>0$. Then it is easy to conclude.
\end{proof}

The understanding of the above degree function is crucial in studying analytic tangent cones. This is first used in \cite{DS} when studying singularities of K\"ahler-Einstein metrics and then was introduced in \cite{CS1} to study singularities of HYM connections. It is proved in \cite{CS1} that if $d(s)<\infty$  then  $s$ gives rise to non-trivial limit homogeneous sections of degree $d(s)$ on all the analytic tangent cones, hence it provides a basic link between $\E$ and $\E_\infty$.  Notice as pointed out in Remark \ref{rmk:sequence} when studying analytic tangent cones we may restrict to a fixed sequence $\lambda_j\rightarrow 0$ given by  
 $$\lambda_j:=2^{-j}.$$
If we denote
$$[s]_j:=\frac{\lambda_j^*s}{\|s\|_j},$$
and
\begin{equation}\label{j-norm}
\|s\|_j:=\|\lambda_j^*s\|_{L^2(B)}=\sqrt{\frac{1}{\Vol{B_{2^{-j}}}} \int_{B_{2^{-j}}}|s|^2\dVol_\omega},
\end{equation} then passing to subsequences we get strong convergence of $[s]_j$ to homogeneous holomorphic sections of degree $d(s)$ on analytic tangent cones. Any such limit $s_\infty$ satisfies $\|s_\infty\|_{L^2(B)}=1$. Again  the arguments in \cite{CS1} were written in the setting of an isolated singularity of $\E$ but tracing the proof one sees that this assumption is not used.   Notice if $d(s)=\infty$, which a priori could be the case,  then we will not be able to obtain anything interesting. Therefore it is important to  find sections $s$ with $d(s)$ finite.

In \cite{CS1, CS2} our idea to study the degree function was to compare the unknown Hermitian metric $H$ with certain explicitly constructed  background Hermitian metric. This allows us to compute the degree explicitly when $\E$ is homogeneous, i.e., $\E=\psi_*\pi^*\underline \E$ for some locally free sheaf $\underline\E$ on $\P^{n-1}$. In general when $\E$ is non-homogeneous or when $\E$ has non-isolated singularities this approach seems to involve very complicated difficulties. In this paper a crucial new observation is that one can directly show finiteness of $d(s)$ for a non-zero $s$, and use this to perform abstract studies without explicit computation of $d(s)$. The main result of this subsection is

\begin{thm}
\label{thm3.1} The following hold
\begin{enumerate}[(1).]
\item For all $s\in \E_0$, we have $d(s)=\infty$ if and only if $s=0$;
\item For all $s\in \E_0$, we have $d(s)\geq 0$; 
\item Given $s, s'\in \E_0$, we have 
\begin{equation}
d(s+s')\geq \min\{d(s), d(s')\};
\end{equation}
\item Suppose there is another admissible HYM connection $(A', H')$ on a reflexive sheaf $\E'$ over $B$, then for $s\in \E_0$ and $s'\in \E'_0$, we have 
\begin{equation}
d(s\otimes s')=d(s)+d(s'), 
\end{equation}
where in each term the degree function has the obvious meaning.
\end{enumerate}
\end{thm}

\begin{proof} The key is Item (1).  First we recall that by Cartan's Theorem A any coherent sheaf over a Stein manifold is generated by global sections. Now by assumption the dual sheaf $\E^*$ is  defined on a neighborhood $\Omega$ of $\overline B$, so it is generated by finitely many global sections over a slightly smaller neighborhood $\Omega'$ of $\overline B$. In 
other words, over $\Omega'$ we have a surjective sheaf homomorphism $\O^{\oplus n_1}\rightarrow \E^*$ for some $n_1$. Applying this again to the kernel sheaf, we then obtain an exact sequence on $B$ of the form \begin{equation}
\O^{\oplus n_2}\rightarrow \O^{\oplus n_1}\rightarrow \E^*\rightarrow 0.
\end{equation}
Taking dual we obtain 
\begin{equation}\label{eqn2.1}
0\rightarrow \E \rightarrow \O^{\oplus n_1} \xrightarrow{\rho} \O^{\oplus n_2}
\end{equation}
We can now endow the natural flat Hermitian metric on $\O^{\oplus n_1}$, then we also get induced Hermitian metrics $H_0$ on $\E$ and $H_0^*$ on $\E^*$ away from $\Sing(\E)=\Sing(\E^*)$. 

\begin{lem}\label{lem3-5}
There exists a constant $C>0$ such that on $B_{1/2}\setminus \Sing(\E)$ we have
\begin{equation}
H\geq C\cdot H_0\end{equation}
\end{lem}
\begin{proof}
By basic linear algebra it suffices to show $H^*\leq CH_0^*$. This then follows from the fact that any holomorphic section $\zeta$ of $\E^*$ over $B$ has $|\zeta|_{H^*}$ uniformly bounded on $B_{1/2}$ (see Theorem $2$ in \cite{BS}). 
\end{proof}

For a holomorphic function $f$ defined on a neighborhood of $0$, we denote by $\deg(f)$ the vanishing order of $f$ at $0$.  Notice if $f$ is not identically zero, then 
$$\deg(f)=\frac{1}{2}\lim_{r\rightarrow 0}\frac{\log \int_{B_r} |f|^2\dVol_{\omega_0}}{\log r}-n.$$
This is an easy consequence using Taylor expansion of $f$ at $0$. We also make the convention that $\deg(f)=\infty$ if $f$ is identically zero.

Given any non-zero $s\in \E_0$, we define 
$$d^0(s)=\frac{1}{2}\lim_{r\rightarrow 0} \frac{\log\int_{B_r} |s|^2_{H_0}\dVol_{\omega_0}}{\log r}-n.$$
Here we emphasize that $d^0$ is defined with respect to the metric $H_0$.
If we view $s$ as a tuple of holomorphic functions $(F_1, \cdots, F_{n_1})$ using \eqref{eqn2.1}, then it is easy to see that if $s\neq 0$, then 
$$
d^0(s)=\min_j \deg(F_j). $$
Lemma \ref{lem3-5} then shows that for nonzero $s$, 
$$d(s)\leq d_0(s)<\infty.$$
 This proves Item (1) of Theorem \ref{thm3.1}.

Item (2) follows from the fact that for any  holomorphic section $s$ of $\E$, $|s|_H$ is locally bounded (c.f. \cite{BS}, Theorem $2$).  Item (3) follows easily from the definition. 

 Now we prove Item (4). Given nonzero $s$ and $s'$, by passing to a subsequence we may assume the rescaled sequences $[s]_j, [s']_j, [s\otimes s']_j$ converges strongly to nonzero homogeneous limit sections $s_\infty, s_\infty', s_\infty''$ respectively. It follows from definition that $[s]_j\otimes [s']_j$ converges strongly to $s_\infty\otimes s_\infty'$. On the other hand, we have 
 $$[s\otimes s']_j=C_j\cdot [s]_j\otimes [s']_j$$
 for some $C_j\in \C^*$. Since $s_\infty, s_\infty', s_\infty''$ are all nonzero, it follows that $|C_j|$ and $|C_j|^{-1}$ are uniformly bounded as $j$ tends to infinity. Passing to a further subsequence we may assume 
 $$s_\infty''=C_\infty s_\infty\otimes s_\infty'$$
 for some $C_\infty\neq 0$. It then follows that 
 $$d(s\otimes s')=d(s_\infty'')=d(s_\infty)+d(s_\infty')=d(s)+d(s').$$
This finishes the proof of Theorem \ref{thm3.1}.
\end{proof}

\begin{cor}
\label{cor3-4}
Denote by $\O_0$ the stalk at $0$ of the sheaf of holomorphic functions on $B$. Then for all  $s\in \E_0$ and  $f\in \O_{0}$, we have 
\begin{equation}
d(fs)=\deg(f)+d(s).
\end{equation}
\end{cor}
\begin{proof}
This follows from  Item (4) above, applied to the case when $\E'$ is the trivial Hermitian line bundle on $B$.
\end{proof}

For our purpose later, we also need the following semi-continuity property of  degrees under taking analytic tangent cones. This property will be crucial in a few places later in this section. 

\begin{prop}\label{prop2.7}
Let $s_j$ be a sequence of holomorphic sections of $\E$ over a fixed neighborhood $B'$ of $0$. Suppose $d(s_j)\geq \mu$ for all $j$, and the rescaled sequence $[s_j]_j$ converges to a nonzero limit section $s_\infty$ on some analytic tangent cone as $j\rightarrow\infty$, then $s_\infty$ (which is not necessarily homogeneous) has degree at least $\mu$. 
\end{prop} 
\begin{proof}
We first make the following 
\begin{clm}\label{clm3.7}
For any $\epsilon>0$ small enough so that $\mu-\epsilon \notin ((\rank\E)!)^{-1}\mathbb{Z}$, there exists an $i_0=i_0(\epsilon)$, so that for any $i\geq i_0$ and $s\in H^0(B',\E)$ with $d(s)\geq\mu$, we  have $\|s\|_i \leq 2^{-(\mu - \epsilon)} \|s\|_{i-1}$. 
\end{clm}
Given this Claim, it follows that $$\|[s_j]_j\|_i\leq 2^{-(\mu-\epsilon)} \|[s_j]_j\|_{i-1}$$ for all $i\geq i_0$.  Taking limit as $j\rightarrow\infty$ we obtain  $$\|s_{\infty}\|_{i} \leq 2^{-(\mu-\epsilon)}\|s_\infty\|_{i-1}$$ for all $i\geq i_0+1$. It follows that $d(s_\infty)\geq \mu-\epsilon$. Letting $\epsilon\rightarrow 0$ we obtain the conclusion. 
\end{proof} 

\begin{proof}[Proof of Claim \ref{clm3.7}]
Otherwise, there exists a subsequence $j_i$ and $s_{j_i} \in H^0(B', \E)$ with $d(s_{j_i})\geq \mu$,  so that for all $i$ large we have $$\|s_{j_i}\|_i> 2^{-(\mu-\epsilon)} \|s_{j_i}\|_{i-1}.$$ On the other hand, by Proposition $3.15$ in \cite{CS1} (again the proof extends trivially to our general setting), we know that there exists some $i_0'$ (depending on $\epsilon$) so that for any $i\geq i_0'$ and $s\in H^0(B_{2^{-i}}, \E)$ if $$\|s\|_i \geq 2^{-(\mu-\epsilon)}\|s\|_{i-1}$$ then $$\|s\|_{i+1} \geq 2^{-(\mu-\epsilon)} \|s\|_i.$$ Then for any $i$ so that $j_i \geq i_0'$, we know $s_{j_i}$ must have degree smaller than $\mu-\epsilon$ which is a contradiction. 
\end{proof}

\subsection{Construction of torsion-free sheaves  on $D$}\label{A torsion-free sheaf on $D$}
In this subsection we define certain canonical torsion-free coherent sheaves on  the exceptional divisor $D=\P^{n-1}$ of the blowup $p:\widehat B\rightarrow B$, which are intrinsically associated to the HYM connection $A$ on $\E$. 

The construction of this subsection can be done using only the stalk $\E_0$, but for the discussion in the next subsection it is more convenient that we work with global sections over $B$ instead of  the stalk $\E_0$. Clearly the degree function defined previously induces a degree function 
\begin{equation}
d: H^0(B, \E) \rightarrow (\rank(\E)!)^{-1}\Z_{\geq0}\cup \{\infty\},
\end{equation}
which satisfies the same properties as those listed in Theorem \ref{thm3.1}.

Let $\mathcal{S}=Im(d)$. We list the nonnegative numbers in $\mathcal S+\mathbb{Z}$ as 
\begin{equation}\label{eqn3.10}
0\leq \mu_1<\cdots <\mu_k<\mu_{k+1}<\cdots.
\end{equation}
We denote by $m$ the biggest integer such that  $\mu_m-\mu_1<1$.  It follows from Corollary \ref{cor3-4} that for $i=1, \cdots, m$ and $l\in \Z_{\geq 0}$, we have $\mu_{i+ml}=\mu_i+l$.

\begin{defi}
For $k\geq 1$, we define 
\begin{equation}
M_k:=\{s\in H^0(B, \E)| d(s)\geq \mu_k\}.
\end{equation}
\end{defi}
It follows from Theorem \ref{thm3.1} (3) that $M_k$ is a $\C$-vector space. 
Since any holomorphic section $s\in H^0(B, \E)$ with $d(s)<\infty$ gives rise to nonzero homogeneous sections on the analytic tangent cones with degree $d(s)$, it follows that on any analytic tangent cone $\E_\infty$, for each $i=1, \cdots, m$, there is a nontrivial direct summand $\underline\E_\infty^{\mu_i}$ of $\underline\E_\infty$ (see Lemma \ref{lem2.6}). However at this moment we do not know if there are possibly other direct summands of $\underline\E_\infty$ since we have not shown how to construct local holomorphic sections of $\E$ from a homogeneous section of $\E_\infty$. Later we will indeed show there are no extra direct summands, see Remark \ref{r:rank equal}.

\begin{prop}\label{l:orthogonal projections}
For each $k$, $M_k/M_{k+1}$ is  finite dimensional as a $\C$-vector space.  
\end{prop}
\begin{proof}
Suppose $M_k/M_{k+1}$ is nontrivial. We choose a sequence of elements $s_{k,1}, s_{k, 2}, \cdots$ in $M_k$ in the following way. We first choose $s_{k,1}$ such that the induced element $\tilde s_{k,1}$  in $M_k/M_{k+1}$ is not zero. Suppose $s_{k, 1}, \cdots, s_{k, l}$ are chosen. If $\tilde s_{k,1}, \cdots, \tilde s_{k, l}$ span $M_k/M_{k+1}$ then we stop. Otherwise  we choose $s_{k,l+1}$ such that $\tilde s_{k, 1}, \cdots, \tilde s_{k, l+1}$ are linearly independent in $M_k/M_{k+1}$. This process a priori may be infinite. In any case if $l\leq \dim_{\C} M_k/M_{k+1}$ then we denote by $\mathcal H_{k,l}$ the $\C$-vector space spanned by $s_{k,1}, \cdots, s_{k, l}$. Then by definition $\mathcal H_{k,l}\cap M_{k+1}=0$, hence $d(s)=\mu_k$ for all $s\in \mathcal H_{k,l}\setminus\{0\}$. 

From these we also construct a sequence of $L^2$ orthonormal sections $\{\sigma^{j}_{k, i}\}$ over $B$ for $j$ sufficiently large (depending on $k$ and $i$), as follows. First we define for $j\gg 1$
$$\sigma^j_{k,1}:=\frac{\lambda_j^* s_{k,1}}{\|\lambda_j^*s_{k, 1}\|_{L^2(B)}}.$$
Suppose $\sigma^j_{k, 1}, \cdots, \sigma^j_{k, l}$ are defined for $j\geq j_0$. Then for $j\gg j_0$ we define $\sigma^j_{k, l+1}$ to be the $L^2$ orthonormal projection of $\lambda_j^*s^j_{k, l+1} $ to the complement of the space spanned by $\sigma^{j}_{k, 1}, \cdots, \sigma^j_{k, l}$, and then normalized to have $L^2$ norm 1. This is the standard Gram-Schmidt process.  

Now fix an analytic tangent cone $\E_\infty$. After passing to a subsequence we may assume for each $i=1, \cdots, l$, $\sigma^j_{k, i}$ converges to a holomorphic section $\sigma^\infty_{k, i}$ on $\E_\infty$ with $\|\sigma^\infty_{k, i}\|_{L^2(B)}\leq 1$, and they are $L^2$ orthogonal over $B$. Now for each $i$, $\sigma^\infty_{k, i}$ is homogeneous of degree $\mu_k$ and $\|\sigma^\infty_{k, i}\|_{L^2(B)}=1$.
Indeed,  for $i=1$ this is simply the fact that $d(s_{k,1})=\mu_k$; for $i\geq 2$ this follows from the same induction argument as Proposition 3.12 in \cite{CS1}.
Given this, it follows that $l$ can not be bigger than the dimension of homogeneous holomorphic sections of degree $\mu_k$ on $\E_\infty$, which is finite by Lemma \ref{lem-finite cone}.

\end{proof}

We denote  
$$n_k=\dim_{\C} M_k/M_{k+1},$$ then the above process stops with $l=n_k$.
For any $\mu\in\{\mu_1,\cdots, \mu_m\}$, we define 
\begin{equation}\label{e:definition of N}N^{\mu}:=\bigoplus_{\mu_k\equiv \mu(\text{mod}\ \mathbb Z)} M_k/M_{k+1}.
\end{equation}
 We define a $\mathbb Z$-grading on $N^\mu$ by setting the degree of $[s]$ to be $\mu_k-\mu$ for $0\neq [s]\in M_k/M_{k+1}$.
 As a direct corollary of Theorem \ref{thm3.1}, we know $N^{\mu}$ is  a graded module over $\C[z_1, \cdots z_n]$.  
  
 The following is the main result of this subsection. 
 
\begin{thm}\label{thm3.8}
For each $\mu\in\{\mu_1, \cdots, \mu_m\}$, $N^\mu$ is a finitely generated torsion-free module over $\C[z_1, \cdots, z_n]$. 
\end{thm}
Before proving this we need some preparation. We fix a given analytic tangent cone $\E_\infty$.
 Let $\{\sigma^j_{k, i}\}$ be the elements constructed as in the proof of Proposition \ref{l:orthogonal projections}. By passing to a subsequence, we may assume for each $k$ and $i$,  $\{\sigma^j_{k,i}\}$ strongly converges to a set of $L^2$ orthonormal homogeneous sections  $\{\sigma^\infty_{k,i}\}$ of $\E_{\infty}$ with degree  $\mu_k$. Suppose $\mu_k=\mu+e$ for $e\in \Z_{\geq 0}$. Then $\{\sigma^\infty_{k, j}\}$ can be viewed as sections in $H^0(\P^{n-1}, \underline \E_\infty^{\mu}(e))$, where $\underline\E_\infty^{\mu}$ is a direct summand of $\underline\E_\infty$ of slope $\mu$. 

Denote by $S_k$ the $\C$-vector space spanned by $\{\sigma^\infty_{k,1}, \cdots, \sigma^\infty_{k,n_k}\}$. It can be viewed as  a subspace of $H^0(\P^{n-1}, \underline\E_\infty^{\mu}(e))$.  By definition  we have 
\begin{equation}\label{dimension equal}\dim_{\C} S_k=n_k=\dim_{\C} M_k/M_{k+1}.
\end{equation}
Define 
\begin{equation}N_{\infty}^{\mu}=\bigoplus_{\mu_k\equiv\mu (\text{mod} \ \Z)} S_k.
\end{equation}
This is also a graded vector space over $\C$ with natural grading given by $\mu_k-\mu$.

\begin{prop}\label{lem2.10}
$N_{\infty}^{\mu}$ is a graded submodule of $\bigoplus_e H^0(\P^{n-1}, \underline\E_\infty^{\mu}(e))$.\end{prop}
\begin{proof}
It suffices to show that for any $k$ fixed, given any homogeneous polynomial $f\in \C[z_1, \cdots, z_n]$ and $\sigma^{\infty}_{k,j}\in S_k$ for any $j$, we have $f\sigma^{\infty}_{k,j} \in S_{k'}$ where $\mu_{k'}=\mu_k+\deg(f)$. By assumption, $f\sigma^{\infty}_{k,i}$ is the limit of $f\sigma^{j}_{k,i}$ as $j\rightarrow\infty$. However,  for each $j$ we have
\begin{equation}\label{eqn-fsigma}f\sigma^{j}_{k,i}-\sum_{i'=1}^{n_{k'}}a_{i'}^j\sigma^j_{k',i'}=s^j_{k'+1}
\end{equation}
 where $\mu_{k'}=\mu_k+\deg(f)$, $a_{i'}^j\in \C$, and $s_{k'+1}^j\in M_{k'+1}$.
 
  We claim $s^j_{k'+1}$ converges to zero in $L^2$ and then the result follows.  Otherwise, suppose the $L^2$ norm of  $\{s^j_{k'+1}\}_j$ has a positive lower bound after passing to a subsequence. Then rescale by factors $A_j\leq C$ we can assume both sides of \eqref{eqn-fsigma} have $L^2(B)$ norm exactly 1 for all $j$. Then it follows that $a_{i'}^j$ is also uniformly bounded for all $i'$. Passing to a subsequence we may assume $a_{i'}^j$ converges to $a_{i'}^\infty$ for all $i'$. Since $\sigma_{k, i}^j$ and $\sigma_{k', i'}^{j}$ converge strongly, and $f$ is fixed, by Lemma \ref{facts about strong convergence} we then obtain that both sides of \eqref{eqn-fsigma} strongly converge to a holomorphic section $s_\infty$ on $\E_\infty$ with $\|s_\infty\|_{L^2(B)}=1$.
 
  Notice by definition for each $j$, $s^j_{k'+1}$ comes from the rescaling of a holomorphic section defined over the fixed ball $B$. Then by Proposition \ref{prop2.7} $s_\infty$ has degree at least $\mu_{k'+1}$. On the other hand, we know $f\sigma^\infty_{k, j}$ is homogeneous of degree $\mu_{k'}$ and by definition the limit $\sigma^\infty_{k', i'}$ is also homogeneous of degree $\mu_{k'}$. This is a contradiction. 
\end{proof}

 \begin{defi} We define   $\underline{\mathcal N}_\infty^\mu$  to be the subsheaf of $\underline\E_\infty^\mu$ generated by $N_\infty^\mu$. 
 \end{defi}
 In particular, $\underline{\N}_{\infty}^{\mu}$ is torsion-free, and for $k\gg1$, $H^0(\P^{n-1}, \underline{\mathcal N}_\infty^\mu(k))$ can be identified with $S_k$. But notice it depends on various choices made above.

\begin{cor}\label{cor-312}
For $k\gg1$, we have $${S_{k'}}=\sum_{l=1}^n z_l {S_k},$$
where $k'$ is such that  $\mu_{k'}=\mu_k+1$. 
\end{cor}
\begin{proof}
This follows from the exact sequence 
\begin{equation}
0\rightarrow \underline{\mathcal F}\rightarrow (\underline\N_\infty^{\mu})^{\oplus n}\xrightarrow{(z_1, \cdots, z_{n})} \underline \N_\infty^{\mu}(1)\rightarrow 0
\end{equation}
for some sheaf $\underline{\F}$ on $\P^{n-1}$. 
Tensoring with $\O(e)$, and noticing that for $e\gg1$, $H^1(\P^{n-1}, \F(e))=0$, we obtain the conclusion. 
\end{proof}

Now we prove Theorem \ref{thm3.8}.
\begin{proof}[Proof of Theorem \ref{thm3.8}]

We first prove the torsion-free property, this follows directly from Corollary \ref{cor3-4}. Indeed,  suppose $[s]\in N^\mu$ is non-zero, then we can write $[s]=\sum_{i\geq i_1} [s_i]$, where $[s_i]\in M_i/M_{i+1}$ and $s_{i_1}\neq 0$. For any nonzero $f\in \C[z_1, \cdots, z_n]$, we write $$f=\sum_{j\geq j_1} f_j,$$ where each $f_j$ is homogeneous and $f_{j_1}\neq 0$.  Let $i'$ be the unique integer such that $\mu_{i'}-\mu_{i_1}=j_1$. Then by Corollary \ref{cor3-4} we know $d(fs)=d(f_{j_1}s_{i_1})=\mu_{i'}$, and the component $[f_{j_1}s_{i_1}]$ of $[fs]$ in $M_{i'}/M_{i'+1}$ is nonzero. 

\

Now we prove
$N^{\mu}$ is finitely generated. It suffices to show that for $k$ large, $$M_{k'}/M_{k'+1}=\sum_l z_l (M_{k}/M_{k+1}), $$
where $k'$ is such that $\mu_{k'}=\mu_k+1$.

\begin{clm}\label{clm3.13}
$\dim_{\C} \sum_l z_l (M_{k}/M_{k+1}) 
\geq \dim_{\C}\sum_l z_l {S_k}.$
\end{clm}
Given this Claim, using Corollary \ref{cor-312} we have 
$$
\begin{aligned}
\dim_{\C} \sum_l z_l (M_{k}/M_{k+1}) 
&\geq \dim_{\C}\sum_l z_l {S_k} \\
&=\dim_{\C}{S_{k'}} \\
&=\dim_{\C} M_{k'}/ M_{k'+1} \\
&\geq \dim_{\C} \sum_k z_l (M_{k}/M_{k+1}), 
\end{aligned}
$$
where the third equation follows by \eqref{dimension equal}. This forces $$\dim_{\C} \sum_k z_l (M_{k}/M_{k+1}) = \dim_\C  M_{k'}/M_{k'+1}.$$ 
Combining this with the fact that  $$\sum_k z_l (M_{k}/M_{k+1})\subset M_{k'}/M_{k'+1},$$ we obtain $$\sum_k z_l (M_{k}/M_{k+1})=M_{k'}/M_{k'+1}.$$ This finishes the proof of Theorem \ref{thm3.8}.  
\end{proof}

\begin{proof}[Proof of Claim \ref{clm3.13}]
Let $\{\sigma_{\infty}^{k',i}\}$ be an $L^2$ orthonormal basis for ${S_{k'}}$ constructed as above. Since we know $$\sum_l z_l {S_{k}}=S_{k'},$$  there exists a sequence of sections $\tau^{j}_{k',i}$ in the $\C$-linear span of $z_l \sigma^{j}_{k, i}$ which converge strongly to $\sigma^{\infty}_{k', i}$. In particular, we have 
\begin{equation}
\lim_{j\rightarrow\infty} \int_{B} \langle \tau^j_{k',i}, \tau^j_{k',l}\rangle =\delta_{il}
\end{equation}
for all $1\leq i, l\leq n_{k'}$.

It suffices to show that $\{\tau_{k',i}^j: 1\leq i\leq n_{k'} \}$ are linearly independent in $M_{k'}/M_{k'+1}$ for $j$ large. We argue by contradiction. Otherwise by passing to a subsequence  we can assume for $j$ large there are constants $a_i^j\in \C$, with $$\sum_{i} a^j_i \tau^{j}_{k',i}=s^j\in M_{k'+1}.$$ 
We normalize $\|s^j\|_{L^2(B)}=1$ which implies $a^{j}_i$ are all uniformly bounded in $j$ since $\{\tau_{k+1,i}^j\}$ are approximately $L^2$-orthonormal for $j$ large enough. In particular, passing to a subsequence we can assume $\{a^{j}_i: 1\leq i\leq n_{k'}\}$ converge to $\{a^{\infty}_i:1\leq i\leq n_{k'}\}$  and there exists some $i$ such that $a_i^{\infty}$ is nonzero. By  Lemma \ref{facts about strong convergence}, we can assume $s^{j}$ converges strongly to some non-zero holomorphic section $s^\infty=\sum_i a^{\infty}_i \sigma^{\infty}_{k',i}$. In particular $d(s^\infty)=\mu_{k'}$. On the other hand, Proposition \ref{prop2.7} implies that  $d(s^\infty)\geq\mu_{k'+1}$. This is a contradiction.
\end{proof}

Theorem \ref{thm3.8} allows us to make the following 

\begin{defi}
For $\mu\in \{\mu_1, \cdots, \mu_m\}$, we define $\underline\E^\mu$ to be the torsion-free sheaf on $\P^{n-1}$ associated to the module $N^\mu$.
\end{defi}

By definition, for $k$ large $H^0(\P^{n-1}, \underline\E^\mu(k))$ can be identified with $M_k/M_{k+1}$. On the other hand,we also know  that by definition for $k$ large  $H^0(\P^{n-1},\underline \N_\infty^\mu(k))$ can be identified with $S_k$. So by \eqref{dimension equal} we have 
\begin{equation}\label{same Hilbert polynomial}
\dim_\C H^0(\P^{n-1}, \underline\E^\mu(k))=\dim_\C H^0(\P^{n-1}, \underline \N_\infty^\mu(k)).
\end{equation}
It follows that $\underline\E^\mu$ and $\underline{\N}_\infty^\mu$ have the same Hilbert polynomial.

Now we recall the asymptotic Riemann-Roch theorem (see Page $189$ in \cite{Kobayashi}) \begin{lem}[Asymptotic Riemann-Roch Theorem] \label{thm1.1}
Let $\underline \F$ be a torsion-free coherent  sheaf over $\P^{n-1}$. Then 
$$
\chi(\underline \F(k))= r\cdot \frac{k^{n-1}}{(n-1)!}+ r(\mu(\underline\F)+\frac{n}{2}) \frac{k^{n-2}}{(n-2)!}+O(k^{n-3}).
$$
where  $\chi$ denotes the holomorphic Euler characteristic, $r$ denotes the rank of $\underline \F$, and $\mu(\underline\F)$ denotes the slope of $\underline\F$.
\end{lem}
An immediate consequence is 
\begin{cor}\label{lem2.12}
We have 
\begin{equation} \label{e:rankequal}\rank(\underline \E^{\mu})=\rank(\underline{\N}_{\infty}^{\mu})
\end{equation}
and 
\begin{equation}\label{e:slopequal}\mu(\underline \E^{\mu})=\mu(\underline{\N}_{\infty}^{\mu}).
\end{equation}
\end{cor}
\begin{rmk}
In the above discussion we work on a fixed  analytic  tangent cone $\E_\infty$, but it is clear that given an analytic tangent cone, by passing to a further subsequence one can extract the a subsheaf $\underline\N_\infty^\mu$ of $\underline\E_\infty^\mu$ as above.  Here ``passing to a further subsequence" is necessary in general since we need the convergence of the chosen holomorphic sections. 
Notice $\underline\E^\mu$ does not depend on the choice of the analytic tangent cone, and the equalities above hold for all such $\underline\N_\infty^\mu$.
\end{rmk}
Now we realize the above defined sheave $\underline \E^{\mu}$ naturally as a factor of the graded sheaf associated to a  filtration of  sheaf on $\P^{n-1}$. For $q\in \Z_{\geq 0}$, $i\geq 1$, we denote
$$M_{q}^{an, i}:=\{s\in H^0(B, \E)| d(s) \geq \mu_i+q\}.$$
For $i=1, \cdots, m$, we then obtain a filtration of $H^0(B, \E)$ given by $\{M_q^{an, i}\}_{q\geq 0}$. We denote the associated graded module by 
$$N^{an, i}:=\bigoplus_{q\geq 0} M^{an,i}_q/M^{an,i}_{q+1}.$$
For $i=1, \cdots, m$ and $l=0, \cdots, m$, we also denote the graded module
$$N^{an, i}_l:=\bigoplus_{q\geq 0}M^{an, m-l+i}_{q}/M_{q+1}^{an, i}.$$
Then we obtain a filtration of graded modules
$$0=N^{an, i}_0\subset N^{an, i}_1\subset \cdots N^{an, i}_m=N^{an, i}.$$
By definition we also know that for $l\geq 1$ the quotient  module $N_l^{an, i}/N_{l-1}^{an, i}$ is isomorphic to the graded module $N^{\mu_{[m-l+i]}}$  (given by \eqref{e:definition of N}) except possibly the lowest degree component. Here we use the notation that
$$
[m-l+i]:=
\begin{cases}
m-l+i \ \ \  \text{if} \ \ \ \  i\leq l;\\
-l+i \ \ \ \ \ \ \ \text{if} \ \ \ \ i >l.
\end{cases}
$$
Then by Theorem \ref{thm3.8} and a simple induction on $l$ it follows that all the graded modules $N_l^{an, i}$ are finitely generated and torsion-free. This enables us to make the following

\

\begin{defi}
For $i=1, \cdots, m$, we  define $\underline{\E}^{an, i}$  to be the torsion-free coherent sheaf on $\P^{n-1}$ associated  to the graded module $N^{an, i}$.
\end{defi}
\begin{defi}
For $i, l=1, \cdots, m$ we  define  $\underline{\E}^{i}_l$ to be the torsion-free coherent sheaf on $\P^{n-1}$ associated to the graded module  $N^{an, i}_l$. 
\end{defi}

\

By definition we have a natural filtration of sheaves given by 
\begin{equation}\label{eqn315}
0=\underline\E_0^i\subset \underline{\E}_1^i \subset 
\cdots \underline{\E}_m^i=\underline{\E}^{an, i}. \end{equation}
Furthermore,  for $i,l=1, \cdots, m$ we have the isomorphism
\begin{equation}\label{eqn320}\underline \E_{l}^i/\underline\E_{l-1}^i\simeq\underline \E^{\mu_{[m-l+i]}}.
\end{equation}
Here we use the fact that two finite generated graded module over $\C[z_1, \cdots, z_n]$ define isomorphic coherent sheaves on $\P^{n-1}$ if and only if they are isomorphic in sufficiently large degrees (see Exercise $5.9$ in \cite{Hartshorne}).

From the above definition we see that the sheaves $\underline \E^{an, i}$ are put into equal footing. In our later discussion when we prove properties of these sheaves we will often restrict to the case $i=1$, and the other cases are just the same up to change of notation. To make notational convenience we also set the following

\begin{notation} 
We make the convention that when we omit the upper script $i$, we always mean $i=1$. So in particular  
$\underline\E^{an}:=\underline\E^{an,1}$, $M^{an}_k:=M^{an, 1}_k$, etc.  
\end{notation}

Notice by definition the sheaves $\underline \E^{an, i}$ on $D=\P^{n-1}$ depend only on the degree function $d$,  and do not depend on choice of analytic tangent cones.  However at this point we can not say much about the geometric properties of either $\underline \E^{an,i}$, or the filtration \eqref{eqn315}. We have only compared the dimension of the space of sections of the quotients associated to the filtration with that of a subsheaf of any analytic tangent cone $\underline\E_\infty$. Also the construction of $\underline\N_\infty^\mu$ depends not only on the analytic tangent cone $\E_\infty$, but also on the choice of holomorphic sections $\{s_{k, i}\}$ at the beginning of this subsection, and it is not a priori clear why this is an intrinsic object.

In the remainder of this section, we will show that $\underline \E^{an,i}$ is an optimal algebraic tangent cone, i.e., it is isomorphic to the restriction to $D$ of some optimal extension of $p^*(\E)|_{\widehat{B} \setminus D}$ across $D$ and the filtration in  \eqref{eqn315} is precisely the Harder-Narasimhan filtration of $\underline{\E}^{an,i}$.

\subsection{The main construction}\label{section3.3}
The main goal of this subsection is to prove 
\begin{thm}\label{thm3.18}
For $i=1, \cdots, m$, there exists a reflexive sheaf $\widehat{\E}^{i}$ on $\widehat{B}$, such that  $\widehat{\E}^i|_{\widehat B\setminus D}$ is isomorphic to $p^*\E|_{\widehat{B} \setminus D}$,  $\widehat{\E}^i|_D$ is isomorphic to $\underline{\E}^{an, i}$, and $H^0(\widehat{B}, \widehat{\E}^i(-kD))$ is naturally identified with $M^{an, i}_k$ for $k\gg 1$.
\end{thm}
\begin{rmk}\label{rmk316}
A priori from the proof below the construction of $\widehat\E^i$ depends on various choices, but later in the next subsection we shall prove each such $\widehat\E^i$ is an optimal extension, so is  indeed unique up to isomorphism by Theorem \ref{optimalalgebraictangentcone}.
\end{rmk}

In the following we shall only prove the case $i=1$, and the arguments for $i\neq 1$ are similar. So we shall omit the superscript $i$ throughout this subsection.

The main difficulty in proving such a statement is that we are working on a mixed situation between algebraic geometry and complex analytic geometry. The exceptional divisor $D$ is algebraic so we can describe sheaves over $D$ in terms of graded modules as in the last subsection. On the other hand, $\widehat{B}$ is not algebraic so it seems not easy to describe $\widehat{\E}$ in terms of purely algebraic objects. To overcome this issue we define an auxiliary sheaf first and then define $\widehat{\E}$ as a subsheaf. 

For our purpose, we need the following vanishing theorem of A. Fujiki (see Theorem $N'$ in \cite{Fujiki})
\begin{lem}\label{lem2.13}
Given any coherent analytic sheaf $\widehat \F$ over $\widehat{B}$, $H^1(\widehat B, \widehat \F (-kD))=0$ for $k\gg1$.
\end{lem} 

As a direct corollary of this, we have 
\begin{lem}\label{lem2.14}
Given any coherent sheaf $\widehat \F$ over $\widehat{B}$, $\widehat{\F}|_D$ is isomorphic to the sheaf associated to the graded module $$\bigoplus_{k\geq 0} H^0(\widehat{B}, \widehat\F(-kD))/H^0(\widehat{B}, \widehat\F(-(k+1)D)).$$
\end{lem}

\begin{proof}
It is a general result in algebraic geometry (see Proposition $5.15$ in \cite{Hartshorne}) that  $\widehat\F|_D$ is the sheaf associated to the graded module 
$$
\bigoplus_{k\geq0} H^0(D,\widehat\F|_D (k)).
$$ By Excersise $5.9$ in \cite{Hartshorne}, it suffices to show that for $k\gg1$  
$$H^0(D, \widehat\F|_D(k))=H^0(\widehat{B}, \widehat\F(-kD))/H^0(\widehat{B}, \F(-(k+1)D)).$$ 
To see this, we use the natural short exact sequence 
$$
0\rightarrow \widehat\F(-(k+1)D) \rightarrow \widehat\F(-kD)\rightarrow \widehat\F|_D(k)\rightarrow 0, 
$$
and the fact that  $H^1(\widehat{B}, \widehat\F(-(k+1)D)=0$ for $k\gg 1$ which follows from Lemma \ref{lem2.13}.
\end{proof}

Now we fix a short exact sequence  given  in   \eqref{eqn2.1}:
\begin{equation}\label{eqn3-8}
0\rightarrow \E \rightarrow \O^{\oplus n_1} \xrightarrow{\rho} \O^{\oplus n_2}.
\end{equation} We denote by $$R:=\O(B)$$ the ring of holomorphic functions over $B$, and we denote by $\m$ the maximal ideal of $R$ consisting of those functions vanishing at $0$. Notice $R$ is not Noetherian. Pulling back \eqref{eqn3-8} to $\widehat B$, we have 
\begin{equation}
0\rightarrow \widehat{\E}_0\rightarrow \O^{\oplus n_1}\xrightarrow{\widehat\rho} \O^{\oplus n_2},
\end{equation}
where we define $\widehat{\E}_0$ to be the kernel of $\widehat\rho$. 

Following the discussion in Section \ref{Finiteness of degree},  we define for $k\in \Z_{\geq 0}$,
$$M^0_{k}:=\{s\in H^0(B, \E): d^0(s) \geq k\}, $$
where $d^0$ is the degree function defined with respect to the fixed induced metric $H_0$.  Then $\{M_k^0\}_k$ forms an $\m$-filtration of $H^0(B, \E)$, that is to say, $$\m\cdot M_k^0\subset M_{k+1}^0$$ for all $k$. Similarly by previous discussion we know $\{M_k^{an}\}_k$ also forms an $\m$-filtration of $H^0(B, \E)$. Since $d(s)\leq d_0(s)$ for all $s$, we have $M_k^{an}\subset M_l^0$ for $l\leq \mu_1+k$. 

 We define the blowup ring $$\widehat{R}:=\bigoplus_{k\geq 0} \mathfrak m^k$$
and the graded modules over $\widehat R$ given by 
$$\widehat{M}^0:=\bigoplus_{k\geq 0} M^0_k,$$
and
\begin{equation}\widehat{M}^{an}:=\bigoplus_{k\geq 0} M^{an}_k.
\end{equation}

Now we make the identification 
$$M_k^0= H^0(\widehat{B}, \Ker(\widehat\rho)(-kD))$$ 
as follows. Given $s\in M_k^0$, then we can view $s$ via \eqref{eqn3-8} as a vector-valued holomorphic function over $B$ which has vanishing order at least $k$ at $0$. Then $p^*s$ is a section of $\Ker(\widehat\rho)$ with vanishing order at least $k$ along $D$. It is easy to see that the converse also holds by Hartog's extension theorem.

\begin{lem}\label{lem3.17'}
$\{M_k^0\}_k$ is  a stable $\m$-filtration of $H^0(B, \E)$, i.e., $$M_{k+1}^0=\m\cdot M_k^0$$ for $k\gg 1$. 
\end{lem}
\begin{proof}
 It suffices to show $M_{k+1}^0\subset \m\cdot M_k^0$. 
Let $\{z_1, \cdots, z_n\}$ denote the coordinate functions on $B$. Then $\{p^*z_i\}_i$ forms a set of global generators of $\O(-D)$. In particular, we have an exact sequence 
$$
0\rightarrow \F \rightarrow  \Ker(\widehat \rho)^{\oplus n} \xrightarrow{(p^*z_1, \cdots, p^*z_n)} \Ker(\widehat \rho)(-D) \rightarrow 0. 
$$
Tensoring with $\O(-kD)$, we have the following exact sequence 
$$
0\rightarrow \F(-kD) \rightarrow  (\Ker(\widehat\rho)(-kD))^{\oplus n} \xrightarrow{(p^*z_1, \cdots, p^*z_n)} \Ker(\widehat \rho)(-(k+1)D) \rightarrow 0. 
$$
By Lemma \ref{lem2.13} above, we have a surjective map 
$$(H^0(\widehat{B}, (\Ker(\widehat\rho)(-kD))))^{\oplus n}\xrightarrow{(p^*z_1, \cdots, p^*z_n)} H^0(\widehat{B}, \Ker(\widehat \rho)(-(k+1)D))$$
for $k$ large. In particular, by  the identification above  we know $M^0_{k+1}\subset \m \cdot M^0_{k}$ and thus $M^0_{k+1}=\m \cdot M^0_k$. 
\end{proof}

\begin{prop}
$\{M_k^{an}\}_k$ is a stable $\m$-filtration of $H^0(B, \E)$. 
\end{prop}

\begin{proof}
 By Theorem \ref{thm3.8}, we know that for $k$ large, $$\m\cdot M_k^{an}\subset M^{an}_{k+1}\subset \m\cdot  M^{an}_{k}+M^{an}_{k+2}.$$  So it suffices to show for any fixed $k$ large, there exists $l_0=l_0(k)$ such that  $$M^{an}_{l} \subset \m \cdot M^{an}_{k}$$ for $l\geq l_0$.  
 Since $$M^{an}_{l} \subset M^0_{l}=\m \cdot M^{0}_{l-1},$$ it suffices to show $M^0_{l-1}\subset M^{an}_{k}$ for $l\gg 1$. Notice by definition we can find $l'$ large such that $$\m^{l'}\cdot H^{0}(B, \E)\subset M_k^{an}.$$ By Lemma \ref{lem3.17'} we may also assume $M_{l+1}^0=\m\cdot M_{l}^0$ for all $l\geq l'$.   Now for $l>2l'+1$, we have $$M^0_{l-1}=\m^{l'}\cdot  M^0_{l-1-l'} \subset \m^{l'}\cdot  M_0^0 \subset M^{an}_k.$$
\end{proof}

\begin{prop}
For all $k$, $M_k^0$ and $M_k^{an}$ are finitely generated $R$ modules. 
\end{prop}
\begin{proof}
 We first show $H^0(B,\E)$ is finitely generated over $R$. By assumption, we know $\E$ is globally generated, so there exists a finite global resolution of $\E$ by sections $s_1, \cdots, s_N \in H^0(B, \E)$ given as 
\begin{equation}\label{eqn3.2}
0\rightarrow \F \rightarrow \O^{\oplus N} \rightarrow \E \rightarrow 0
\end{equation} 
for some coherent sheaf $\F$. Since $B$ is Stein, we know $H^1(B, \F)=0$ and thus we have a surjective map $$H^0(B, \O)^{\oplus N } \xrightarrow{(s_1, \cdots, s_N)} H^0 (B, \E)\rightarrow 0.$$ In particular, we know $H^0(B, \E)$ is finitely generated over $R$. 

Now we consider $M_k^0$. Let $\E^k\subset \E$ be the coherent subsheaf generated by $M_k^0$ in $\E$. By definition, for all $k$, we have  $$\m^k \cdot H^0(B, \E) \subset M_k^0.$$
This implies that $\mathcal I_0^k \cdot \E \subset \E^k$. In particular, $\{z_i^k s_j: 1\leq i\leq n, 1\leq j\leq N\}$ globally generates $\E^k$ away from $0$. Now take finitely many sections of $M_k^0$ which generate the stalk of $\E^k$ at $0$, then combined with $\{z_i^k s_j: 1\leq i\leq n, 1\leq j\leq N\}$, they globally generate $\E^k$. So it follows that for any section $s$ of $\E^k$ we have $d^0(s)\geq k$. Therefore  $H^0(M, \E^k)=M_k^0$. It follows  from the above argument using vanishing of $H^1$ that $M_k^0$ is finitely generated over $R$.

Again the same argument also works with $M_k^{an}$, noticing that for all $k$, $M_k^{an}\supset \m^{l}\cdot H^0(B, \E)$ for $l=k+\mu_1+1$.  
\end{proof}

An immediate corollary is 

\begin{cor}\label{stable}
Both $\widehat{M}^0$ and $\widehat {M}^{an}$ are finitely generated over $\widehat R$. 
\end{cor}
\begin{proof}
This follows from Proposition $5.3$ in Eisenbud \cite{Eisenbud}. Notice this result does not require the ring $\widehat R$ to be Noetherian. 
\end{proof}

Now we finish the proof of Theorem \ref{thm3.18}.
We divide it into a few pieces.

\

\noindent \textbf{Existence}:

\

 By Corollary  \ref{stable}, $\widehat{M}^{an}$ is finitely  generated over $\widehat R$. We can choose a set of finitely many homogeneous generators $\{[s_{k_i}]\in M^{an}_{k_i}/M^{an}_{k_i+1}\}_{i=1}^{l}$, and the corresponding representatives $s_{k_i}\in  M^{an}_{k_i} \setminus M^{an}_{k_i+1}$. In particular, we have 
 $$d(s_{k_i})=k_i+\mu_{q}$$
for some $q\in \{1, \cdots, m\}$. Given $s\in M^{an}_{k_i}$, we know $$d^0(s) \geq d(s)\geq k_i+\mu_1 \geq k_i,$$ thus $s\in M^{0}_{k_i}$. In particular, we have  $$M^{an}_{k_i}\subset M^{0}_{k_i}$$ and $$s_{k_i} \in H^0(\widehat B, \Ker(\widehat{\rho})(-k_iD)),$$ hence $s_{k_i}$ defines a map $\O(k_i D) \rightarrow \Ker(\widehat{\rho})$. We define $\widehat{\E}$ to be the image sheaf of the natural map $$\bigoplus_{i=1}^l \mathcal O(k_iD) \xrightarrow{(s_{k_1}, \cdots, s_{k_l})} \Ker(\widehat \rho)$$ In particular, $\widehat{\E}$ is coherent and it lies in the following exact sequence 
\begin{equation}\label{equation3.27}
0\rightarrow \widehat{\F} \rightarrow\bigoplus_{i=1}^l \mathcal{O}(k_i D) \rightarrow \widehat{\E} \rightarrow 0,
\end{equation}
for some sheaf $\widehat \F$.

\

\noindent \textbf{Global sections:}

\

 We show that $H^0(\widehat{B}, \widehat{\E}(-kD))=M^{an}_k$ for $k$ sufficiently large. Notice both are naturally subspaces of $H^0(B, \E)$. We first show that
 $$
 H^0(\widehat{B}, \widehat{\E}(-kD)) \subset M^{an}_k.
 $$
 By using the short exact sequence (\ref{equation3.27}), since $H^1(\widehat{B}, \widehat{\F}(-kD))=0$ for $k\gg1$,  we have a surjective map 
$$\bigoplus_{i=1}^l H^0(\widehat{B},\mathcal{O}((k_i-k) D)) \xrightarrow{(s_{k_1}, \cdots, s_{k_l})} H^0(\widehat{B}, \widehat{\E}(-kD))\rightarrow 0$$
for $k$ large. In particular, for any section $s\in H^0(\widehat{B}, \widehat{\E}(-kD))$, 
$$s=\sum_i f_{k-k_i} \cdot s_{k_i}$$
 for some $f_{k-k_i}\in \m^{k-k_i}$.  This implies $$d(s)\geq \mu_1+k,$$ so $s\in M^{an}_k$. It remains to show that $M^{an}_k \subset H^0(\widehat B, \widehat{\E}(-k))$. By our choices of $\{s_{k_i}\}$ it follows that given any $s'\in M^{an}_k$, we can write 
 $$s'=\sum_i g_{k-k_i} \cdot s_{k_i}$$ where $g_{k-k_i} \in \m^{k-k_i}$. In particular, we know $s' \in H^0(\widehat{B}, \widehat{\E}(-kD))$. 
 
 \

\noindent\textbf{Restriction}:

\

$\widehat\E|_D=\underline{\E}^{an}$ this follows from Lemma \ref{lem2.14} and the previous item.

\

\noindent\textbf{Reflexivity:} 

\

By definition $\widehat \E$ is a subsheaf of $\Ker (\widehat\rho)$ and we have a natural inclusion 
$$0\rightarrow \widehat{\E}^{**} / \widehat{\E} \rightarrow \Ker(\widehat\rho)/\widehat{\E}.$$ 
 The above  inclusion induces an inclusion of global sections 
$$0\rightarrow  H^0(\widehat{B}, (\widehat{\E}^{**}/\widehat{\E}) (-kD)) \rightarrow H^0(\widehat{B}, (\Ker(\widehat\rho)/\widehat{\E})(-kD))$$ 
for any $k$. By the discussion above, we have $$H^0(\widehat{B}, \Ker(\widehat \rho)(-kD))=M_k^0$$ and $$H^0(\widehat{B}, \widehat{\E}(-kD))=M^{an}_k$$ for $k$ large. In particular, by Lemma \ref{lem2.13} we have
$$H^0(\widehat{B}, (\Ker(\widehat\rho)/\widehat{\E}) (-kD))=M^0_k/M^{an}_{k}$$ 
for $k$ large. Similarly, 
\begin{equation}\label{eqn322}
H^0(\widehat{B}, (\widehat{\E}^{**}/\widehat{\E}) (-kD))=H^0(\widehat{B}, \widehat{\E}^{**}(-kD))/M^{an}_k.
\end{equation}

\begin{clm}\label{clm3.23}
For any fixed $k\gg 1$, given any $s\in H^0(\widehat{B}, \widehat{\E}^{**}(-kD))$, there exists a homogeneous polynomial function $P$ over $\C^n$ so that $$p^*P\cdot s\in H^0( \widehat{B}, \mathcal I_D^{d}\cdot \widehat{\E}(-kD)),$$ where $d=\deg(P)$, and $\mathcal I_D$ denotes the ideal sheaf of $D$ on $\widehat B$.
\end{clm}
Given this, fix any $s\in H^0(\widehat{B}, \widehat{\E}^{**}(-kD))$, we have $$p^*P\cdot s\in H^0(\widehat{B}, \mathcal I_D^{d}\cdot \widehat{\E}(-kD))=M^{an}_{k+d}.$$
 Notice $s$ naturally induces a section of $\E$ on $B$, and we have then  $$d(P\cdot s)\geq \mu_1+k+d.$$ 
 Thus by Corollary \ref{cor3-4} we get $$d(s)\geq \mu_1 +k,$$
hence $s\in M^{an}_k$. By \eqref{eqn322}, for any $k$ large, we have $$H^0(\widehat{B}, (\widehat{\E}^{**}/\widehat{\E}) (-kD))=0.$$ This implies $\widehat{\E}=\widehat{\E}^{**}$, in other words, $\widehat{\E}$ is reflexive. This finishes the proof of Theorem \ref{thm3.18}.

\begin{proof}[Proof of Claim \ref{clm3.23}]
Let $\tau_k=(\widehat{\E}^{**}/\widehat{\E}) (-kD)$. Since $\widehat{\E}$ is torsion-free, the support $V \subset D \subset \widehat{B}$ of $\tau_k$ is of complex codimension at least $2$ in $\widehat B$.  We have the following observation 

\

\textit{
At any point $z\in V$, there exists a meromorphic function over $\widehat{B}$ of the form $$f=\frac{p^*P}{p^*Q}$$ which vanishes at $z$ and $f\cdot (\tau_k)_z=0$,  where $(\tau_k)_z$ denotes the stalk of $\tau_k$ at $z$, and $P, Q$ are homogeneous polynomials on $\C^n$. 
}

\

In particular, we see $$(p^*P\cdot  s)_z\in (\mathcal{I}_D^{deg(P)}\cdot \widehat{\E}(-kD))_z,$$ Since $V$ is compact, we can find finitely many such $p^*P_1, \cdots, p^*P_l$ so that for any $z\in D$, there exists some $P_i$ so that $$(p^*P_i \cdot s)_z \in (\mathcal I_D^{d_i}\cdot\widehat{\E}(-kD))_z.$$ Here $d_i=\deg(P_i)$. If we denote $\widehat{P}=p^*(P_1\cdots P_l)$, it follows that for any $z\in D$, there exists some $i$ such that 
$$
\begin{aligned}
(\widehat{P}\cdot s)_z 
&\in p^*(P_1 \cdots P_{i-1} P_{i+1} \cdots P_{l})_z (p^{*}P_i \cdot \widehat{\E}(-kD))_z \\
&\subset p^*(P_1 \cdots P_{i-1} P_{i+1} \cdots P_{l})_z (\mathcal{I}_D^{d_i}\cdot \widehat{\E}(-kD))_z\\
&\subset (\mathcal{I}_D^{d}\cdot \widehat{\E}(-kD))_z.
\end{aligned}
$$
Here $d=d_1+\cdots d_l.$ This finishes the proof of Claim \ref{clm3.23}.

Now we justify the observation above. By definition,  there exists some $N$ so that  the annihilator ideal sheaf $\text{Ann}(\tau_k)$ satisfies
$$(\mathcal{I}_V)^N \subset \text{Ann}(\tau_k) \subset \mathcal{I}_V,$$
where $\mathcal I_V$ denotes the ideal sheaf of $V$ on $\widehat B$. 
It suffices to show that there exists a meromorphic function over $\widehat{B}$ of the form $f=\frac{p^*P}{p^*Q}$ so that $f$ vanishes along $V$ near $z$. To see this, we let $\mathcal{I}_{V,D}$ denote the ideal sheaf associated to $V$ in $D$. We know $\mathcal{I}_{V,D}(l)$ is globally generated for $l$ large. Fix such an $l$, since $V$ has complex codimension at least $1$ in $D$, there exists some $P'\in H^0(D, \mathcal{I}_{V,D}(l))$ and $Q'\in H^0(D, \mathcal{O}(l))$ so that $P'$ is not identically zero  and $Q'(z) \neq 0$. Then $\phi^*(\frac{P'}{Q'})$ will be what we need. Here $\phi: \widehat{B} \rightarrow D$ denote the restriction of the projection map $\O(-1) \rightarrow D$. Indeed, it is direct to check that $(\phi^*(\frac{P'}{Q'}))=p^*(\frac{P}{Q})$ where $P,Q$ are the homogeneous polynomials on $\C^n$ corresponding to $P'$ and $Q'$ respectively.   
\end{proof}

For each $i=1, \cdots, m$, we define the sheaf on $D$ $$\underline{\mathcal{T}}^{\mu_i}:= \underline\E_\infty^{\mu_i}/\underline{\N}_{\infty}^{\mu_i}$$ and we denote $$\underline{\mathcal{T}}:=\bigoplus_{i=1}^m \underline{\mathcal{T}}^{\mu_i}.$$
\begin{cor}\label{corollary3.24}
$\underline{\mathcal{T}}$ is a torsion sheaf. Moreover, for all $i=1, \cdots, m$, we have $$\mu(\underline{\E}^{\mu_i}) \leq \mu_i,$$ and the equality holds if and only if the support of $\underline{\mathcal{T}}^{\mu_i}$ has complex codimension at least $2$. 
\end{cor}

\begin{proof}
By Lemma \ref{thm1.1} and Corollary \ref{lem2.12}, we already know that $$\rank(\underline{\E}^{\mu_i})=\rank(\underline{\N}_{\infty}^{\mu_i}) \leq \rank(\underline\E_\infty^{\mu_i}) .$$ On the other hand, by Theorem \ref{thm3.18} and the filtration (\ref{eqn315}), we have 
$$\sum_{i=1}^m \rank(\underline{ \E}^{\mu_i})=\rank(\underline \E^{an}) = \rank(\E)\geq \sum_{i=1}^m \rank(\underline{\E}_\infty^{\mu_i}).$$
This forces the inequality above to be an equality, so $\underline{\mathcal T}$ is a torsion sheaf. By definition, we also have 
$$
0\rightarrow\underline{\mathcal N}_\infty^{\mu_i}\rightarrow \underline\E_\infty^{\mu_i}\rightarrow \underline{\mathcal T}_\infty^{\mu_i}\rightarrow 0,
$$
which implies $$c_1(\underline\E_\infty^{\mu_i})=c_1(\underline{\mathcal N}_\infty^{\mu_i})+c_1(\underline{\mathcal T}_\infty^{\mu_i}).$$ Let $V$ be the  closure of the codimension $1$ part of the support of $\underline{\mathcal T}_\infty^{\mu_i}$. We know $c_1(\underline{\mathcal T}_\infty^{\mu_i})$ is equal to the Poincar\'e dual of $V$ (see Proposition $3.1$ in \cite{SW}). In particular, we have $\mu(\underline{\E}^{\mu_i}) \leq \mu_i$ and the equality holds if and only if  $c_1(\underline{\mathcal T}_\infty^{\mu_i})=0$, i.e., the support of $\underline{\mathcal{T}}^{\mu_i}$ has complex codimension at least $2$. 
\end{proof}
\begin{rmk}\label{r:rank equal}
In particular, we also know that  $$\underline\E_\infty=\bigoplus_{i=1}^m \underline\E_\infty^{\mu_i}. $$
In other words, if we write 
$$\underline {\mathcal  N}_\infty:=\bigoplus_{i=1}^m\underline{\mathcal N}_\infty^{\mu_i}$$  then
$$\underline{\mathcal T}=\underline\E_\infty/\underline{\mathcal N}_\infty.
$$
Furthermore, we know 
\begin{equation}\underline \N_\infty^{**}=\underline{\E}_\infty
\end{equation}
 which is an important fact to be used later.
\end{rmk}
In the following subsection, we will show that the support of $\underline{\mathcal{T}}$ has complex codimension at least $2$, and as a consequence we see that $\widehat \E$ is an optimal extension.

\subsection{Optimality}\label{Optimality}
 For $i=1, \cdots, m$, let $\widehat \E^i$ be the reflexive  sheaf constructed in Theorem \ref{thm3.18}.  The goal of this subsection is to prove 
\begin{thm}\label{prop-optimal}
 $\widehat{\E}^i$ is an optimal extension of $\E$ at $0$. Moreover,  the Harder-Narasimhan filtration of $\underline{\E}^{an, i}$ is
given by \eqref{eqn315}:
\begin{equation}
0=\underline\E_0^i\subset \underline{\E}_1^i \subset 
\cdots \underline{\E}_m^i=\underline{\E}^{an, i},\end{equation} 
and the associated graded sheaf 
$\Gr^{HN}(\underline \E^{an, i})$  is isomorphic to $\bigoplus_{i=1}^m\underline \E^{\mu_i}$.
\end{thm}

Let $\G^{alg}$ and $\Sigma_b^{alg}$ be defined in Corollary \ref{cor2-18}. Then an immediate consequence is

\begin{cor}\label{cor325}
We have
\begin{equation}
\mathcal G^{alg}=\psi_*\pi^*\bigoplus_{i=1}^m(\Gr^{HNS}(\underline\E^{\mu_i}))^{**}, 
\end{equation}
and 
\begin{equation}
\Sigma_b^{alg}=\sum_{i=1}^m\Sigma_b^{alg}(\underline\E^{\mu_i}). 
\end{equation}
\end{cor}
To prove Theorem \ref{prop-optimal}, again we only prove the case $i=1$ and the other cases are similar. So we shall omit the superscript $i$ throughout this subsection. By \eqref{eqn320} we know for $l=1, \cdots, m$
\begin{equation}
\underline \E^{l}/\underline\E^{l-1}=\underline\E^{\mu_{m-l+1}}. 
\end{equation}
So Theorem \ref{prop-optimal} follows from the definition and the following two Propositions. 

\begin{prop}\label{prop325}
For $i=1, \cdots, m$, we have  $\mu(\underline\E^{\mu_i})=\mu_i$. In particular, the support of the torsion sheaf $\underline{\mathcal T}$ has codimension at least 2. 
\end{prop}
\begin{prop}\label{prop326}
For $i=1, \cdots, m$, $\underline\E^{\mu_i}$ is semi-stable. 
\end{prop}
In the following we shall prove these two results. By Corollary \ref{corollary3.24} we already know $\mu(\underline\E^{\mu_i})\leq \mu_i$. To prove Proposition \ref{prop325} it suffices to prove the reversed inequality. The key idea is to make use of the dual sheaf $\E^*$, which is endowed with the induced HYM connection. We can apply the previous construction in this section to obtain a natural algebraic tangent cone of $\E^*$. We will show this algebraic tangent cone is naturally dual to the algebraic tangent cone $\widehat \E$ of $\E$, and applying \ref{corollary3.24} yields the desired reversed inequality. This duality should be viewed as a manifestation of the fact that taking dual is an intrinsic operation, both algebraically and analytically. Below we give the detailed arguments.

By construction we know $\widehat\E$ satisfies the following
\begin{itemize}
\item[(1)] $\widehat{\E}|_{\widehat{B} \setminus D} \simeq (p^*\E)|_{\widehat{B} \setminus D}$, and $H^0(\widehat B, \widehat \E(-kD))$ is naturally identified with $M_k^{an}\subset H^0(B, \widehat \E)$ for  $k$ large;
\item[(2)] there exist finitely many sections $s_i\in M_{k_i}^{an}\setminus M_{k_i+1}^{an}$,  $i=1, \cdots, l$,  which globally generate $\widehat{\E}$ in the following sense
\begin{equation}\label{eqn3.11}
 \bigoplus_{1\leq i \leq l} \mathcal O(k_iD) \xrightarrow{(s_1, \cdots, s_l)} \widehat{\E} \rightarrow 0
\end{equation}
where $s_i$ is naturally viewed as an element in $H^0(\widehat B, \widehat \E(-k_iD))$; 
\item[(3)] there exists a filtration of $\underline\E^{an}:=\widehat{\E}|_D$  
$$
0=\underline{\E}_0\subset \underline{\E}_1 \subset \cdots \underline{\E}_m=\underline\E^{an} 
$$
so that $\mu(\underline{\E}_{i+1}/ \underline{\E}_{i}) \leq \mu_{m-i}$ and $\rank(\underline{\E}_{i+1}/ \underline{\E}_{i})=\rank(\underline{\E}_\infty^{\mu_{m-i}})$;
\item[(4)] the support of the torsion sheaf $\underline{\mathcal{T}}$ has complex codimension at least $2$ if and only if $\mu(\underline{\E}_{i+1}/ \underline{\E}_{i})=\mu_{m-i}$ for any $1\leq i\leq m$.
\end{itemize}
Similar construction applies to the dual admissible HYM connection $A^*$ on $\F:=\E^*$.   Abusing notation, we still denote the degree function associated to $\F$ by $d$. It is clear that given an analytic tangent cone $\E_\infty$ of $A$ at $0$, then $\F_\infty:=\E^*_\infty$ is an analytic tangent cone of $A^*$. It follows that for any nonzero section $s^*\in H^0(B, \F)$, $d(s^*)\equiv -\mu_i (\text{mod} \ \mathbb{Z})$ for some $i\in \{1, \cdots, m\}$. For any $k'\in \mathbb{Z}$, we denote
$$L^{an}_{k'}=\{s^*\in H^0(B,\F)| d(s^*) \geq k'-\mu_1\}.$$
Then $\hat{\F}$ satisfies
\begin{itemize}
\item[$(1')$] $\widehat{\F}|_{\widehat{B} \setminus D} \simeq  (p^*\F)|_{\widehat{B} \setminus D}$, and $H^0(\widehat B, \widehat \F(-k'D))$ is naturally identified with $L_{k'}^{an}\subset H^0(B, \widehat \F)$ for all $k'$ large;
\item[$(2')$]
there exist finitely many sections $s_{i'}^*\in L_{k_{i'}}^{an}\setminus L_{k_{i'}+1}^{an}$,  $i'=1, \cdots, l'$,  which globally generate $\widehat{\F}$ in the following sense
\begin{equation}\label{eqn3.12}
\bigoplus_{1\leq i \leq l'} \mathcal O(k_{i'}D) \xrightarrow{(s_1^*, \cdots, s_{l'}^*)} \widehat{\F} \rightarrow 0
\end{equation}
where $s_{i'}^*$ is naturally viewed as an element in $H^0(\widehat B, \widehat {\F}(-k_{i'}D))$;

\item[$(3')$] There exists a filtration of $\underline\F^{an}:=\widehat{\F}|_D$
$$
0=\underline\F_0\subset \underline{\F}_1 \subset \cdots \underline{\F}_m=\underline\F^{an} 
$$
so that $\mu(\underline{\F}_{i'+1}/ \underline{\F}_{i'}) \leq -\mu_{i'}$ with $\rank(\underline{\F}_{i'+1}/ \underline{\F}_{i'})=\rank(\underline \F_\infty^{-\mu_i'})=\rank(\underline\E^{\mu_{i'}}_\infty)$; 
\item[$(4')$] the support of the torsion sheaf $\underline{\mathcal T'}$ has complex codimension at least $2$ if and only if $\mu(\underline{\F}_{i'+1}/ \underline{\F}_{i'})=-\mu_{i'}$ for any $1\leq i'\leq m$.
\end{itemize}

\begin{lem}\label{prop3.25}
We have
$$\mu(\underline{{\E}}^{an})+\mu (\underline{\F}^{an})\leq 0,$$
 and the equality holds if and only if the support of $\underline{\mathcal T}$ and $\underline{\mathcal T'}$ have complex codimension at least $2$.  
\end{lem}
\begin{proof}
By $(3)$ and $(3')$ above, we have 
$$
\begin{aligned}
&\mu(\underline{{\E}}^{an})+\mu (\underline{{\F}}^{an})\\
=&\sum_{i=1}^m \frac{ \rank(\underline{\E}_{i+1}/ \underline{\E}_{i}) \mu(\underline{\E}_{i+1}/ \underline{\E}_{i})}{\rank(\E)}+\sum_{i'=1}^{m} \frac{\rank(\underline \F_{i'+1}/ \underline\F_{i'}) \mu(\underline\F_{i'+1} / \underline\F_{i'})}{\rank(\E)}\\
=& \sum_{i=1}^m \frac{\rank(\underline\E^{\mu_{m+1-i}}_\infty )  \mu(\underline{\E}_{i+1}/ \underline{\E}_{i})}{\rank(\E)}+\sum_{i'=1}^{m} \frac{\rank(\underline\E^{\mu_{i'}}_\infty )  \mu(\underline\F_{i'+1} / \underline\F_{i'})}{\rank(\E)}\\
\leq & \sum_{i=1}^m \frac{ \rank(\underline\E^{\mu_{m+1-i}}_\infty ) \mu_{m+1-i} }{\rank(\E)}-\sum_{i'=1}^{m} \frac{\rank(\underline\E^{\mu_{i'}}_\infty )  \mu_{i'}}{\rank(\E)}\\
=& 0. 
\end{aligned}
$$
The equality holds if and only if $\mu(\underline{\E}_{i+1}/ \underline{\E}_{i}) =\mu_{m+1-i}$ and $\mu(\underline{\F}_{i'+1}/ \underline{\F}_{i'}) = -\mu_{i'}$ for all $1\leq i,i'\leq m$. By Corollary \ref{corollary3.24}, this holds if and only if the support of $\underline{\mathcal T}$ and $\underline{\mathcal T}'$  has complex codimension at least $2$.
\end{proof}

\begin{lem}\label{lem3.17}
There exists a natural paring  
$$ L^{an}_{k'} / L^{an}_{k'+1} \times M^{an}_k/M^{an}_{k+1}  \rightarrow \mathfrak{m}^{k+k'}/\mathfrak{m}^{k+k'+1}$$
which sends $([s^*_{k'}], [s_{k}])$ to $[s^*_{k'}(s_{k})]$. Furthermore, this paring is non-degenerate for $k$ and $k'$ both large.
\end{lem}
\begin{proof}
The existence of the paring follows from Theorem \ref{thm3.1}, Item (4). It remains to prove this paring is non-degenerate. Suppose there exists some nonzero $[s^*]\in L^{an}_{k'} / L^{an}_{k'+1}$ so that $s^*(t)\in \mathfrak{m}^{k'+k+1}$ for all $t\in M^{an}_k$ for $k$ large. Let $s^*_\infty$ be the nonzero rescaled limit section of $\F_\infty$ given by $s^*$ and $t_\infty$ be the nonzero rescaled limit section of $\E_\infty$ given by $t$. Taking the limit of the fact $s^*(t)\in \mathfrak{m}^{k'+k+1}$, we know $s^*_\infty(t_\infty)=0$. By Corollary \ref{corollary3.24}, for $k$ large, we know that such limiting homogeneous sections $t_\infty$ generate the fiber of $\E_\infty$ at a generic point of $\C^n$. In particular, we know $s^*_\infty$ has to vanish at a generic point hence $s_\infty^*=0$. Contradiction. 
\end{proof}

\begin{lem}\label{prop3.28}
There exists a natural paring  $\widehat{\E}\otimes \widehat{\F}\rightarrow \O_{\widehat B}$ induced by the paring $\E\otimes\F\rightarrow \O_B$.  Furthermore, this paring induces an isomorphism between $\widehat{\E}^*$ and $\widehat{\F}.$
\end{lem}

\begin{proof}
We first define  the natural paring $\widehat{\E}\otimes \widehat{\F}\rightarrow \O_{\widehat B}$. This follows from \eqref{eqn3.11} and \eqref{eqn3.12}. More precisely, we first define a paring $$
P: \bigoplus_i \O(k_i) \times \bigoplus_{i'} \O(k_{i'}) \rightarrow \O_{\widehat{B}}
$$
which sends $((f_1\cdots, f_l), (g_1\cdots g_{l'}))$ to $\sum_{i,i'} f_i g_{i'} s^*_{i}(s_{i'})$. Here $s^*_{i}(s_{i'})\in \mathfrak{m}^{k_i+k_{i'}}$ by Theorem \ref{thm3.1}, Item (4), and we naturally view it as a holomorphic function defined over $\widehat{B}$. In particular, $s^*_{i}(s_{i'})$ has a vanishing order at least $k_i+k_{i'}$ along $D$ and thus $f_i g_{i'} s^*_{i}(s_{i'})$ is a well-defined local homorphic function. It is now easy to see that this paring descends to be a paring between $\widehat{\E}$ and $\widehat{\F}$, which we also denote  as $P$. Furthermore, by definition, away from $D$, this paring is naturally isomorphic to the paring between $\E$ and $\F$ over $B^*$. 
\begin{clm}\label{clm3.27}
 $P$ induces a vector bundle isomorphism at a generic point when restricting to the exceptional divisor $D$.
\end{clm}
Given this claim, we know $P$ induces an isomorphism between $\widehat \E$ and $\widehat{\F}$ away from a codimension $2$ analytic subvariety of $\widehat{B}$. Since $\widehat{\E}$ and $\widehat{\F}$ are reflexive, $P$ actually induces a global isomorphism. 
\end{proof}
\begin{proof}[Proof of Claim \ref{clm3.27}]
For any $k$ and $k'$, $P$ induces a paring
$$
\widehat{\E}(-kD) \times \widehat{\F}(-k'D) \rightarrow \O(-(k+k')D).
$$
 Let $\underline P$ denote the induced paring on $D$
  $$\underline{{\E}}^{an}(k) \times \underline{{\F}}^{an}(k') \rightarrow \O(k+k').$$ By definition, $\underline P$ induces the paring of global sections
 $$ L^{an}_{k'} / L^{an}_{k'+1} \times M^{an}_k/M^{an}_{k+1}  \rightarrow \mathfrak{m}^{k+k'}/\mathfrak{m}^{k+k'+1}$$
which sends $([s^*_{k'}], [s_{k}])$ to $[s^*_{k'}(s_{k})].$ Indeed, this follows from the construction that $$H^0(D, \underline{{\E}}^{an}(k))=M^{an}_k/M^{an}_{k+1}$$ and $$H^0(D, \underline{{\F}}^{an}(k'))=L^{an}_{k'}/L^{an}_{k'+1}$$ for $k$ and $k'$ large. Furthermore, this paring  is non-degenerate for $k$ and $k'$ large. Let 
$$\underline P^*: \underline{{\E}}^{an}(k) \rightarrow (\underline{{\F}}^{an}(k'))^* \otimes \O((k'+k)D)=(\underline{{\F}}^{an})^*(k)$$ 
be the map induced by $\underline P$. We know that $\underline P^*$ induces an injective map of global sections 
$$
\underline P^*: H^0(D, \underline{{\E}}^{an}(k)) \rightarrow H^0(D, (\underline{{\F}}^{an})^*(k))
$$
for any $k$ large. In particular, by Lemma \ref{thm1.1}, the rank of the image sheaf $\text{Im}(\underline P^*)$ has rank equal to $\rank(\underline{{\E}}^{an})$ which implies $\underline P$ induces an isomorphism at a generic point.  
\end{proof}

\begin{proof}[Proof of Proposition \ref{prop325}]
By Lemma \ref{prop3.28} we have $(\widehat \E)^*\simeq \widehat \F$, so away from a codimension 2 subvariety in $D$, we have $(\underline \E^{an})^*\simeq\underline \F^{an}$. It follows that $\mu(\underline\E^{an})+\mu(\underline\F^{an})=0$.   Then the conclusion follows from Lemma \ref{prop3.25} and Corollary \ref{corollary3.24}.
\end{proof}

\begin{proof}[Proof of Proposition \ref{prop326}]
We argue by contradiction. Suppose $\underline\E^{\mu_i}$ is not semistable, then there exists a subsheaf $\underline{\F}$ of $\underline{\E}^{\mu_i}$ which is stable and satisfies $\mu(\underline{\F})>\mu_i$. Take a basis $\{[s_k]\}^N_{k=1}$ of $H^0(\P^{n-1}, \underline \F(i_0)) \subset M_{i'_0}/M_{i'_0+1}$ where $s_{k}\in M_{i_0'} \setminus M_{i_0'+1}$ for $k=1,\cdots N$. Here $i_0$ is chosen large enough so that $\underline \F(i_0)$ is globally generated and $\mu_{i_0'}=i_0+\mu_i$. In particular, we have an exact sequence
$$
\O^{\oplus N} \xrightarrow{([s_1], \cdots, [s_N])} \underline \F(i_0) \rightarrow 0.
$$
Let $$m_j=\max_{k} \|s_k\|_j.$$ Fix an analytic tangent cone $\E_\infty$, and passing to a further subsequence of $\{j\}$ if necessary we may assume $m_j^{-1}\cdot \lambda_j^*s_k$ converges strongly to a homogeneous section $s_k^\infty$ of $\E_\infty$ for all $k$, and at least one of the $s_k^\infty$ is non-zero.  We define a nontrivial sheaf homomorphism  
$$Q:\O^{\oplus N} \rightarrow \E_{\infty}$$ 
by sending $(a_1,\cdots, a_N)$ to $\sum a_k s_k^{\infty}$.  Also we know $d(s_{k}^\infty)=\mu_i+i_0$ for any $k$ with $s_k^\infty\neq0$. We want to show $Q$ descends to be a nontrivial sheaf homomorphism from $\underline \F(i_0)$ to $\underline \E^{\mu_i}_{\infty}(i_0)$. 
It suffices to show that if $$\sum_k a_k[s_k]([z])=0$$ then $$\sum_k a_ks_k^{\infty}|_{\C^* . z}=0.$$ 
\begin{clm}\label{clm3.30}
Given $[z]\notin \Sing(\underline{{\E}}^{an}) \cup \Sigma\cup Z(\E)$, where $\Sigma$ is the bubbling set of the convergence to the analytic tangent cone $\E_\infty$, if $$\sum_k a_k[s_k]([z])=0,$$ then $$\sum_k a_k s_k|_{\C^*\cdot z \cap B} = s|_{\C^* \cdot z \cap B}$$ for some $s\in H^0(B, \E)$ with $d(s)>i_0+\mu_i$.
\end{clm}
Given this Claim, we have $$ \frac{\sum_k a_k (\lambda_j^*s_k)(z)}{m_j}= \frac{(\lambda_j)^*s (z)}{m_{j}}=\frac{\|s\|_j}{m_j} \cdot \frac{(\lambda_j)^*s (z)}{\|s\|_j}$$
By definition 
$$\lim_{j\rightarrow\infty} \frac{\log m_j}{-j\log 2}=2(i_0+\mu_i+n)$$
whereas 
$$\lim_{j\rightarrow\infty} \frac{\log \|s\|_j}{-j\log 2}=2(d(s)+n)$$

In particular, $$\lim_{j\rightarrow\infty} \frac{\|s\|_j}{m_j}=0$$  
It follows that $\sum_k a_ks_k^{\infty}|_{\C^* . z}=0$.
So $Q$ descends to be a nontrivial sheaf homomorphism from $\underline \F(i_0)$ to $\underline \E^{\mu_i}_{\infty}(i_0)$ away from $ \pi(\Sigma\cup Z(\E)) \cup \Sing(\underline{{\E}}^{an})$ which is of complex codimension at least 2. It then extends to a nontrivial map from $\underline \F(i_0)$ to  $\underline \E^{\mu_i}_{\infty}(i_0)$ over the entire $D$.  However, since  $\underline \F(i_0)$ is stable and  $\underline \E^{\mu_i}_{\infty}(i_0)$ is polystable with $\mu(\underline \F(i_0))>\mu_i+i_0=\mu(\underline \E^{\mu_i}_{\infty}(i_0))$, such a map can not be non-trivial. This is a contradiction. 
\end{proof}

\begin{proof}[Proof of Claim \ref{clm3.30}]
Without loss of generality we may assume $$[z]=((0,0,\cdots, 0), [1,0,\cdots 0])\in \widehat B\subset B\times \P^{n-1}.$$ Then the map $p$ is locally given by $$p(z_1, \omega_2, \cdots \omega_n)=(z_1, z_1\omega_2, \cdots, z_1 \omega_n).$$ By choosing a local trivialization of $\widehat{\E}(-i_0D)$ near $[z]$ given by \emph{global} sections, we can view the sections locally as vector valued holomorphic functions. Near $[z]$, by  Taylor expansion, we have $$ \sum_k a_ks_k(z)=z_1 f_1 + \omega_2 f_2 + \cdots \omega_n f_n.$$
 Let $\C_{[z]}$ be the line given by $\omega_2=\cdots \omega_n=0$, which can be identified with the line $z_2=\cdots=z_n=0$ in $\C^n$ through the projection map $p$. Then $$\sum_k a_ks_k(z)|_{\C_{[z]}}= z_1 f_1|_{\C_{[z]}}.$$
Let $f$ be the pull-back of $f_1|_{\C_{[z]}}$ under the composition map $\widehat{B} \rightarrow B \xrightarrow{\rho} \C_{[z]}$, where $\rho$ is the natural orthogonal projection with respect to the flat metric.  By our choice of local trivialization we may view $f$ as a section $s'$ of $\widehat \E(-i_0D)$. Let $s=z_1s'$, then we have $$\sum_k a_k s_k|_{\C^* . z \cap B} = s |_{\C^* . z \cap B}.$$ 
  Also $$d(s)=d(s')+1\geq i_0+\mu_1+1>i_0+\mu_i$$ because $\mu_i-\mu_1<1$.  
\end{proof}

\begin{rmk}\label{freechoice2}
From Remark \ref{rmk316} a priori the definition of $\widehat\E^i$ is not canonical, but since we know it is optimal and the restriction $\widehat{\E}^i|_D$ is isomorphic to $\underline \E^{an, i}$ which is intrinsically defined, by Theorem \ref{optimalalgebraictangentcone} we know that the isomorphism class of $\widehat\E^i$ is uniquely determined. Moreover, 
we have obtained $m$ optimal extensions $\widehat\E^{i}(i=1, \cdots, m)$.   By Theorem \ref{optimalalgebraictangentcone},  we know they are related by Hecke transforms of special type, and    our  constructions can recover all the different optimal algebraic tangent cones up to equivalence.
\end{rmk}

\section{Proof of the main theorem}\label{Proof of the main theorem}
In this Section we prove Theorem \ref{main}. In Section \ref{uniquenessoflimitingsheaf} we prove part (I), and in Section \ref{the bubbling set} we prove Part (II). 
Throughout this section we fix an analytic tangent cone $(A_\infty, \Sigma_b^{an})$ with underlying reflexive sheaf $\E_\infty$, which arises as the rescaled limit corresponding to a subsequence of the fixed sequence $\{\lambda_j=2^{-j}\}$. We will use the algebro-geometric results in \cite{GT} and \cite{GSRW}, which is very different from the pointwise orthogonal projection technique developed in \cite{CS2}. The overall discussion is very similar to the arguments in \cite{GSRW}. 

\label{section4}
\subsection{The limiting sheaf and connection}\label{uniquenessoflimitingsheaf}
By Lemma \ref{lem2.6} and Corollary \ref{cor325}, Part (I) of Theorem \ref{main} is  a consequence of the following 

\begin{prop}\label{prop4.1}
For $l=1, \cdots, m$,  we have 
$$\underline\E_\infty^{\mu_l}\simeq (\Gr^{HNS}(\underline \E^{\mu_{l}}))^{**}, $$
and 
 $$\Sigma_b^{alg}(\underline \E^{\mu_l})=\mathcal C(\underline{\E}_\infty^{\mu_l}/\underline{\N}_\infty^{\mu_l}).$$
\end{prop}
 Notice by \eqref{eqn320} in Section \ref{A torsion-free sheaf on $D$} we know for $i=1, \cdots, m$, $\underline\E^{\mu_{m-i+1}}$ is identified with $\underline \E_1^{i}$, which is the first term of in the Harder-Narasimhan filtration of $\underline \E^{an, i}$. We shall only prove the case $i=1$, which corresponds to the case $l=m$ in the above Proposition. For general $l$ one can get the conclusion  by repeating the argument with $\underline \E^{an, i}$ replaced by $\underline \E^{an, m-l+1}$. As before we shall omit  the superscript $i$ and only consider $\underline\E^{an}=\underline\E^{an, 1}$.

 Denote by $\widehat{\E}$  the corresponding optimal extension constructed in Section \ref{section3.3}. Then we know  for $k\gg1$ there is a natural identification
$$H^0(\widehat{B}, \widehat{\E}(-kD))\simeq M_k^{an}=\{s\in H^0(B, \E): d(s) \geq \mu_1+k\}.$$ Suppose the Harder-Narasihman filtration of $\underline \E^{an} = \widehat{\E}|_D$ is given by 
$$0=\underline\E_0\subset \underline \E_1\subset \cdots \underline \E_m=\underline \E^{an}, $$
where $\mu(\underline \E_{l}/\underline \E_{l-1})=\mu_{m+1-l}$. Then $$\underline \E_{l}/\underline \E_{l-1}=\underline \E^{\mu_{m+1-l}}$$ is the sheaf associated to the graded module $N^{\mu_{m+1-l}}$ (see \eqref{e:definition of N}).

Now for $k\gg1$ we know
 $\widehat{\E}(-kD)$ is globally generated and we have the following exact sequence 
$$H^0(\widehat{B}, \widehat{\E}(-kD))\rightarrow H^0(D, \underline{\E}^{an}\otimes \O_D(k))\rightarrow0$$
given by restriction to $D$.
We may also assume the sheaf $\underline\E_1(k)$ is globally generated. 
Choose sections  $s_{i} \in H^0(\widehat{B}, \widehat{\E}(-kD))=M_k^{an}$, $i=1, \cdots, N$, so that when restricting to $D$ they form a basis $\{\underline s_i\}$ of the vector space  $H^0(D, \underline\E_1(k))$, and they globally generate the sheaf $\underline \E_1(k)$. 

As in section \ref{A torsion-free sheaf on $D$}, for each $j$, we can perform the Gram-Schmidt process to the sections $\lambda_j^*s_i$ over $B$, and obtain sections $\sigma_{i}^j$, $i=1, \cdots, N$, which are $L^2$-orthonormal over $B$.
 Passing to a further subsequence if necessary we may assume the sections $\sigma_{i}^j$ converge strongly to holomorphic sections $\sigma_i^\infty$ of $\E_\infty$. The latter are all homogeneous of degree $k+\mu_m$ and they induce sections $\{\underline\sigma_i^\infty\}_{i=1}^N$ of the sheaf $\underline{\N}_{\infty}^{\mu_m}(k)$ on $D$.  By definition, for $k\gg1$ we know $\underline{\N}_{\infty}^{\mu_m}(k)$ is globally generated by these sections.
 
 Now we fix  a $k\gg1$ so that all the above named properties are satisfied. Notice $\{\sigma_{i}^j\}_{i=1}^N$ can also be viewed as sections of $\lambda_j^*\widehat{\E}(-kD)$ over $\widehat B$, and they differ from $\{\lambda_j^*s_{i}\}_{i=1}^N$ be an element in $GL(N; \C)$. In particular we have a sequence of sheaf homomorphisms over $\widehat B$ given by

\begin{equation}\label{exact sequence 0}
\O_{\widehat B}^{\oplus N} \xrightarrow{q_j=(\sigma_{1}^j, \cdots, \sigma_{N}^{j})} \lambda_j^*\widehat{\E}(-kD). 
\end{equation}
Now we want to take limits of these as $j$ tends to infinity. Over $\widehat B$ there are no natural limits, but we can restrict to either $D$ or $\widehat B\setminus D\simeq B\setminus\{0\}$ to take limits. The first limit is \emph{algebraic} and the second limit is \emph{analytic}.  It turns out these two limits can be related,  and we will use the algebraic results quoted in Section \ref{section2.3} to get the conclusion about the analytic limit.

On the one hand, by the choice of $k$, we know that when restricting to $D$, \eqref{exact sequence 0} yields an exact sequence
$$
\O_D^{\oplus N} \xrightarrow{\underline q_j:=(\underline{\sigma}_{1}^j, \cdots, \underline \sigma_{N}^{j})} \underline{\E}_1(k) \rightarrow 0.
$$
In particular we obtain a sequence of points in $\Quot(\mathcal H, \tau)$, where $\tau$ denotes the Hilbert polynomial of $\underline \E_1$. By passing to a subsequence, we can take an algebraic limit 
\begin{equation}\label{limit1}\O_D^{\oplus N} \xrightarrow{\underline q_\infty^{alg}=(\underline{\sigma}_{1}^{alg,\infty}, \cdots, \underline \sigma_{N}^{alg,\infty})} \underline{\E}^{alg}_\infty(k) \rightarrow 0
\end{equation}
  in $\Quot(\mathcal{H}, \tau)$. Notice the notation here is different from Section \ref{section2.3} in that we have tensored everything by $\O(k)$. Also notice a priori $\underline\E_\infty^{alg}$ may not be torsion-free, and may depend on the choice of subsequences. 

Fixing a smooth Hermitian metric on $\O_D^{\oplus N} $, we can identity the map $\underline q_j$ with the projection map $\underline \pi_j$ to the orthogonal complement of $\Ker(\underline q_j)$ away from $\Sing(\underline\E_1)$. We may similarly identify $\underline q_\infty$ with $\underline \pi_\infty$ outside $\Sing(\underline\E_\infty^{alg})$ . We also fix a smooth Hermitian metric on the locally free part of $\underline{\E}_1$, so that the adjoint $\underline q_j^*$ is well-defined. 
  From Lemma \ref{quoteschemeanalyticconvergence}, we have the following 

\begin{lem}\label{lem4.1}
 $\underline{\pi}_j$ converge to $\underline \pi_\infty$ smoothly away from $\Sing(\underline \E^{alg}_\infty)\cup \Sing(\underline \E_1)$. 
\end{lem}
For our purpose, we also note the following elementary formula
 $$\underline{\pi}_j=\underline q^*_j (\underline q_j \underline q^*_j)^{-1}\underline q_j$$ 
 away from $\text{Sing}(\underline \E_1)$.

On the other hand, we can restrict the exact sequence \eqref{exact sequence 0} to $B\setminus \{0\}$, 
\begin{equation}
\O_{B\setminus\{0\}}^{\oplus N} \xrightarrow{q_j=(\sigma_{1}^j, \cdots, \sigma_{N}^{j})} \lambda_j^*\E. 
\end{equation}
Now by our discussion above each $\sigma_{i}^j$ converges to $\sigma_{i}^\infty$, which is a degree $\mu_m+k$ homogeneous section of $\E_\infty$. So we obtain the exact sequence
\begin{equation}\label{limit2}\O_D^{\oplus N} \xrightarrow{\underline q^{an}_\infty:=(\underline \sigma_1^\infty, \cdots,  \underline\sigma_N^\infty)} \underline \N_\infty^{\mu_m}(k) \rightarrow 0.
\end{equation}
Our goal is to compare \eqref{limit1} with \eqref{limit2}. In the arguments below we need to consider the dual sheaf $\E^*$, endowed with the dual HYM metric $H^*$. 
Let $\widehat{\E^*}$ be the optimal extension of $\E^*$ at $0$  constructed in  Section \ref{Optimality}. Then there exists a non-degenerate paring between $\widehat{\E}$ and $\widehat{\E^*}$ induced by the paring between $\E$ and $\E^*$ over $B$. For $k'\gg1$, we may find a set of sections $\{\overline{\sigma}_{i}^{j}\}_{i=1}^{N'}$ in $H^0(\widehat{B}_j, (\widehat{\E^*})_j(-k'D))$ (where $(\widehat{\E^*})_j:=\lambda_j^*(\widehat{\E^*})$), which are $L^2$ orthonormal over $B$ with respect to the rescaled metric $\lambda_j^* H^*$. Moreover, they converge strongly to homogeneous limits $\{\overline{\sigma}_i^\infty\}_{i=1}^{N'}$, which globally generate $\F_\infty=\E_\infty^*$ away from a complex codimension 2 subvariety $\mathcal S$ (by Proposition \ref{prop325}).   

\begin{prop}\label{prop4.3}
There exists an isomorphism $$\underline{\rho}_\infty: \underline{\E}^{alg}_\infty(k) \rightarrow \underline{\N}^{\mu_m}_\infty(k)$$ induced by $\underline q_\infty^{alg}$ and $\underline q_\infty^{an}$,  i.e., $\underline\rho_\infty$ sends $\sum _j a_j \underline\sigma_j^{ alg, \infty}$ to $\sum_j a_j \underline \sigma_j^{\infty}$. In particular, $\underline \E^{alg}_\infty$ is torsion-free.
\end{prop}

\begin{proof}
It suffices to show that there exists a sheaf inclusion $$\Ker(\underline q^{alg}_\infty)\subset \Ker(\underline q^{an}_\infty)$$ away from a proper subvariety of $D$. Indeed, given this, $\underline{q}_\infty^{an}(\Ker(\underline q^{alg}_\infty))$ will be a torsion subsheaf of $\underline{\N}_\infty^{\mu_m}$ and has to be zero since $\underline \N_\infty^{\mu_m}$ is torsion-free. This shows that $\underline\rho_\infty$ is well-defined. By definition $\underline\rho_\infty$ is surjective, so must be an isomorphism. This is because the two sheaves $\underline\E_\infty^{alg}$ and $\underline{\mathcal N}_\infty^{\mu_m}$ have the same Hilbert polynomial by \eqref{same Hilbert polynomial} and the definition of the Quot scheme. Take a generic point $[z] \in D$ such that the following hold
\begin{itemize}
\item $\underline\E_1$ and $\underline \E_{\infty}^{alg}$ are locally free near $[z]$ and the convergence $\underline\pi_j\rightarrow\underline\pi_\infty$ is smooth. 
\item there exists a neighborhood $U$ of the line $\C_{[z]}\subset \C^n\cap B$, such that for all $\delta>0$, $\E$ and $\E_\infty$ are locally free on $U\setminus B_\delta$ and
$U\cap (\Sigma\cup Z(\E))=\{0\}$ where $\Sigma$ is the bubbling set (see \eqref{bs}) and $Z(\E)$ is the $\C^*$ invariant reduced subvariety in $\C^n$ underlying the Zariski tangent cone of $\text{Sing}(\E)$ at $0$.
\item $\widehat\E$ is locally free in a neighborhood of $[z]$. 
\item $\underline \N_\infty^{\mu_m}$ is locally free in a neighborhood of $[z]$, and its fiber over $[z]$ is generated by the limit sections $\underline\sigma_{1}^\infty, \cdots, \underline\sigma_{N}^\infty$.
\item $\C_{[z]}\cap \mathcal S=\{0\}$.
\end{itemize}
Over the point $[z]$, it suffices to show the inclusion   $\Ker(\underline q^{alg}_\infty)|_{[z]}\subset \Ker(\underline q^{an}_\infty)|_{[z]}$ as vector subspaces of $\C^N$ (the fiber of $\O_D^{\oplus N}$ at $[z]$). 
Suppose $a\in \C^{\oplus N}$ is given  so that $\underline q_\infty^{alg}(a)=0$. Then by definition we have $\underline \pi_\infty(a)|_{[z]}=0$. Hence by Lemma \ref{lem4.1} $$\lim_{j\rightarrow\infty} \underline\pi_j(a)|_{[z]}=0.$$ 
Fix a nonzero $z\in \C^n$ on the complex line $\C_{[z]}$.  We want to show 
$$\lim_{j\rightarrow\infty} q_j(a)(z)=0, $$
where this limit is taken in the analytic sense.

Now we let
$$\sigma^j= q_j(a-\underline\pi_j(a)|_{[z]}). $$  It follows that $\sigma^j([z])=0$. Moreover, the analytic limit on the analytic tangent cone $\E_\infty$ is given by  $$\sigma^\infty:=\lim_{j\rightarrow\infty} \sigma^j=\lim_{j\rightarrow\infty} q_j(a).$$  It suffices to show that $\sigma^\infty$ vanishes at any $z\neq 0$ on the line $\C_{[z]}$. 

To show this we make use of the dual sheaf $\E^*$. Notice by our choice of $[z]$, $\underline\F_\infty$ is locally free in a neighborhood of $[z]$, so we may choose $r=\rank(\E)$ sections, say $\overline\sigma_1^\infty, \cdots,\overline\sigma_r^\infty$, which generate the fibers of $\F_\infty$ over $\C_{[z]}\setminus\{0\}$. 

For each $i$, we have 
\begin{equation}\label{Eqn4.1}
\overline{\sigma}_i^j(\sigma^j)=f_{i}^j \in H^0(\widehat{B}, \mathcal I_{[z]}(-(k+k')D)).
\end{equation}
In particular, when viewed as a holomorphic function over $B$, $f_i^j|_{\C_{[z]}}$ has a vanishing order at least $k+k'+1$. If we take limits to the analytic tangent cones, we know that $\overline\sigma_i^\infty$ has degree $k'$ and $\sigma^\infty$ has degree $k$, so the function
 $$f_i^{\infty}=\overline{\sigma}_i^\infty(\sigma^\infty)$$ is a homogeneous polynomial which  either is identically zero, or has degree exactly equal to $k+k'$. In particular, restricting to the line $\C_{[z]}$, $f_i^{\infty}|_{\C_{[z]}}$ either vanishes identically or has a vanishing order exactly equal to $k+k'$ at $0$. However, since for each $j$, $f_i^j|_{\C_{[z]}}$ vanishes at $0$ up to order at least ${k'+k+1}$, and they converge uniformly to $f_i^\infty$ outside any small neighborhood of $0$, by convergence of residues, we know $f_i^{\infty}|_{\C_{[z]}}$ vanishes at $0$ up to order at least $k'+k+1$ . In particular, for any $i$, $f_i^\infty$ is identically zero on the line ${\C_{[z]}}$. So $$\overline{\sigma}^\infty_i(\sigma^\infty)|_{\C_{[z]}}=0.$$ By assumption, $\{\overline{\sigma}_i^\infty\}_{i=1}^{r}$ generate the fibers of $\E_\infty^*|_{\C_{[z]}}$, so we conclude that $\sigma^\infty|_{\C_{[z]}}=0$.  

\end{proof}

Now we need to use the results from Section \ref{section2.3} and finish the proof. Indeed, since $\underline{\E}^{\mu_m}_\infty=(\underline{\mathcal{N}}_{\infty}^{\mu_m})^{**}$ admits an admissible HYM connection, it is polystable, Proposition \ref{prop4.1} follows directly from Corollary \ref{S-equivalence}.   

\begin{rmk}\label{rmk4.7}
In the proof above, to find each factor $\underline{\E}_\infty^{\mu_i}$, we can choose a set of global sections $\{s_{i,l} \in H^0(B, \E)\}_{l=1}^{N_i}$ with  $d(s_{i,l})=k+\mu_i$ (here we can assume that $k\gg1$ is uniform for all $i$), such that  the homogeneous limits given by those sections generate $\underline{\N}_\infty^{\mu_i}(k)$. Furthermore, by assumption, the union of all $s_{i,l} (i=1, \cdots, m, l=1, \cdots, N_i)$  spans $M^{an}_k / M^{an}_{k+1}$ over $\C$, and they globally generate $\widehat{\E}(-kD)$ over $\widehat{B}$.  For simplicity we shall denote these sections by $\sigma_1, \cdots, \sigma_N$, and the induced sections of $\lambda_j^*\E$ by $\sigma_1^j, \cdots, \sigma_N^j$. 
\end{rmk}

\subsection{The analytic blow-up cycle}\label{the bubbling set}

By Corollary \ref{cor325} and Proposition \ref{prop4.1}, to prove Part (II) of Theorem \ref{main}, it suffices to prove 

\begin{equation} \label{e:cycle equal}
\sum_{l=1}^m\mathcal C(\underline{\E}_\infty^{\mu_l}/\underline{\N}_\infty^{\mu_l})=\Sigma_b^{an}.
\end{equation}
By definition we have
$$\underline\E_\infty=\bigoplus_{l=1}^m \underline\E_\infty^{\mu_l}.$$
As in Remark \ref{r:rank equal}, we denote 
$$\underline {\mathcal N}_\infty:=\bigoplus_{l=1}^m \underline\N_\infty^{\mu_l}, $$
and
$$\mathcal N_\infty:=\psi_*\pi^*\underline{\mathcal N}_\infty.$$
By Remark \ref{rmk4.7}, one can find a sequence of global resolutions of $\E_j=\lambda_j^*\E$ over $B\setminus\{0\}$, which are of the form 
$$0\rightarrow \mathcal{K}_j  \rightarrow  \O_{B\setminus\{0\}}^{\oplus N} \xrightarrow{p_j=(\sigma^{j}_1,\cdots, \sigma^j_N)} \E_j \rightarrow 0. $$
 By passing to a subsequence we may assume the above converges to a global resolution of $\N_\infty\subset \E_\infty$: \begin{equation}\label{equation5.1}
0\rightarrow \mathcal{K}_\infty \rightarrow \O_{B\setminus\{0\}}^{\oplus N} \xrightarrow{p_\infty=(\sigma^{\infty}_1,\cdots, \sigma^\infty_N)} \N_\infty \rightarrow 0.
\end{equation}
 Here the crucial fact for us is that by homogeneity Equation (\ref{equation5.1}) naturally descends to the projective space $\P^{n-1}$ as  a short exact sequence
\begin{equation}\label{equation5.2}
0\rightarrow \underline {\mathcal{K}}_\infty \rightarrow \O_{\P^{n-1}}^{\oplus N} \xrightarrow{\underline p_\infty=(\underline\sigma^{\infty}_1,\cdots \underline\sigma^\infty_N)} \underline \N_\infty(k) \rightarrow 0.
\end{equation}
This  enables us to use the global Bott-Chern formula in \cite{SW, GSRW} to calculate the algebraic multiplicity of the torsion sheaf $\underline{\mathcal T}=\underline\E_\infty/\underline\N_\infty$ along a codimension two irreducible subvariety in $\P^{n-1}$.

 In the following we shall fix a flat Hermitian metric on $\O_D^{\oplus N}$ which induces a flat metric on $\O_{{B\setminus\{0\}}}^{\oplus N}$.  Let ${\mathcal{K}}_j^{\perp}$ denote the orthogonal complement ${\mathcal{K}}_j^{\perp}$ of ${\mathcal{K}}_j$ in $\O^{\oplus N}$ with respect to the flat metric. It is a smooth sub-bundle of $\mathcal O^{\oplus N}$ away from the $\Sing(\E_j)$. We denote by $p_j^*A_j$ the induced connection on ${\mathcal{K}}_j^{\perp}$ through the smooth isomorphism ${\mathcal{K}}_j^{\perp}\simeq \E_j$.  Similarly we define $p_\infty^*A_\infty, \underline p_\infty^*\underline A_\infty$ etc. 

Given \emph{any} codimension two irreducible subvariety $\underline \Gamma\subset \P^{n-1}$, the algebraic multiplicity $m^{alg}(\underline\Gamma)$ of $\underline\E_\infty/\underline \N_\infty$ along $\underline \Gamma$ is defined to be $$m^{alg}(\underline \Gamma):=\dim H^0(\underline \Delta, \mathcal T|_{\underline \Delta})$$ where $\underline \Delta$ is a  holomorphic transverse slice of $\underline \Gamma$ at a generic point.  Then by definition we have 
$$\mathcal C(\underline{\mathcal T})=\sum_{\underline\Gamma} m^{alg}(\underline\Gamma)\cdot[\underline\Gamma].$$

Take any such $\underline \Gamma$. 
We can compute $m^{alg}(\underline\Gamma)$ using the exact sequence (\ref{equation5.2})  (see \cite{SW, GSRW}). To explain this, we fix a flat unitary connection $\underline {A_f}$ on $\O_D^{\oplus N}$ hence a flat unitary connection $A_f$ on $\O_{B\setminus\{0\}}^{\oplus N}$. Let $G_j$ be the induced connection on ${\mathcal{K}}_j$ and $G_\infty$ be the pull-back of a fixed admissible connection $\underline G_\infty$ on $\underline {\mathcal{K}}_\infty$. Here being admissible means that $\underline G_\infty$ is a smooth connection on the locally free locus of $\underline{{\mathcal{K}}}_\infty$ which has finite Yang-Mills energy  and bounded $\Lambda_{\omega_{ FS}} F_{\underline G_\infty}$. Notice since ${\mathcal N}_\infty$ is torsion-free, we know ${\mathcal{K}}_\infty$ hence $\underline {{\mathcal{K}}}_\infty$ is reflexive, which implies its singular set has complex codimension at least 3. So we can make the connection $\underline G_\infty$  smooth at a generic point of $\underline\Gamma$.

We know that $A_\infty$ descends to a HYM connection on $\underline \E_\infty(k)$, which is the direct sum of admissible HYM connections. We denote it by $\underline A_\infty$. 
Choose a two dimensional smooth  subvariety $\underline \Delta$ of $\P^{n-1}$ so that $\underline \Delta$ intersects $\underline\Gamma$ transversely and positively at $\underline z_1,\cdots, \underline z_l$.  Then we have (see Page $59$ in \cite{GSRW})
\begin{equation}\label{equation5.3}
\begin{aligned}
&l\cdot  m^{alg}(\underline\Gamma)\\=&\frac{1}{8\pi^2} \sum_i\{\int_{ \underline \Delta \cap  B_\epsilon(\underline z_i)}Tr(F_{\underline {A_f}} \wedge F_{\underline {A_f}})-Tr(F_{\underline G_\infty\oplus (\underline p_\infty)^* \underline A_\infty} \wedge F_{\underline G_\infty\oplus (\underline p_\infty)^* \underline A_\infty})\\
&-\int_{\partial( \underline \Delta \cap B_\epsilon(\underline z_i))}CS(\underline {A_f}, \underline G_\infty\oplus (\underline p_\infty)^* \underline A_\infty)\}
\end{aligned}
\end{equation}
where $B_{\epsilon}(\underline z_i)$ denotes a small ball of radius $\epsilon$ in $\P^{n-1}$, and $CS(\cdot, \cdot)$ is the Chern-Simons functional defined using the obvious trivialization of $\O_{\P^{n-1}}^{\oplus N}$ over the boundary of each $\underline\Delta\cap B_{\epsilon}(z_i)$ (see \cite{CS2}, Section $4.2$).  Here we may assume each $\underline\Delta\cap B_\epsilon(\underline z_i)$ only intersects $\underline\Gamma$ at one point and also $$(B_\epsilon(\underline z_i)\setminus \underline\Gamma)\cap (\text{supp}(\Sigma^{an}_b)\cup \Sing(\underline A_\infty))=\emptyset$$
Notice \eqref{equation5.3} was stated in \cite{GSRW} when $\underline\Gamma$ is contained in the support of $\underline{\mathcal T}$. But if $\underline\Gamma$ is not contained in the support of $\underline{\mathcal T}$, then $m^{alg}(\underline\Gamma)=0$ and \eqref{equation5.3} obviously holds.

For each $i$, we choose a holomorphic lift $\Delta_i\subset B$ of $\underline \Delta \cap B_\epsilon(\underline z_i)$. Then as a direct corollary of Equation (\ref{equation5.3}), we obtain  (see similar calculation in Section $4.2$ in \cite{CS2}) 
\begin{equation}
\label{equation5.4}
\begin{aligned}
l\cdot  m^{alg}(\underline\Gamma)=&\frac{1}{8\pi^2}\sum_i \{ \int_{ \Delta_i }Tr(F_{{A_f}} \wedge F_{{A_f}})-Tr(F_{ G_\infty\oplus  p_\infty^* A_\infty} \wedge F_{ G_\infty\oplus p_\infty^* A_\infty})\\
&-\int_{  \partial \Delta_i}CS({A_f},  G_\infty\oplus  p_\infty ^*A_\infty)\}
\end{aligned}
\end{equation}

On the other hand, let $\Gamma\subset\C^n$ be the affine cone over $\underline\Gamma$, then the analytic multiplicity along $\Gamma$ is given by 
$$l\cdot m^{an}(\Gamma):=\lim_{j\rightarrow\infty} \frac{1}{8\pi^2}\sum_i\{\int_{\Delta_i}Tr(F_{A_{j}} \wedge F_{A_{j}})-Tr(F_{A_\infty}\wedge F_{A_\infty})\}.$$
So we have 
$$\Sigma_{b}^{an}=\sum_{\underline\Gamma} m^{an}(\Gamma)\cdot [\underline\Gamma].$$
Now the equality \eqref{e:cycle equal} follows from 
\begin{prop} For any codimension two irreducible subvariety $\underline\Gamma\subset \P^{n-1}$, we have   
$m^{alg}(\underline\Gamma)=m^{an}(\Gamma).$
\end{prop}
\begin{proof}
By Equation (\ref{equation5.4}) and Lemma \ref{analyticmultiplictyformula}, a direct computation yields
\begin{eqnarray*}
&&l\cdot m^{alg}(\underline\Gamma)-l\cdot m^{an}(\Gamma)\\ &=& \sum_i \{\int_{\partial \Delta_i} CS(G_\infty \oplus p_\infty^*A_\infty, G_{j}\oplus p_j^* A_{j})\\ && +\frac{1}{8\pi^2}\int_{\Delta_i} Tr(F_{G_j} \wedge F_{G_j})-Tr(F_{G_\infty} \wedge F_{G_\infty})\}.
\end{eqnarray*}
It suffices to show that for each $i$
\begin{equation}\label{equation5.5}
\lim_{j\rightarrow\infty} \{\int_{\partial \Delta_i} CS(G_\infty \oplus p_\infty^* A_\infty,G_{j}\oplus p_j^* A_{j})+\int_{\Delta_i} Tr(F_{G_j} \wedge F_{G_j})-Tr(F_{G_\infty} \wedge F_{G_\infty})\}=0.
\end{equation}
By assumption, using the given flat metric, we obtain an orthogonal splitting 
$$\O_{B\setminus\{0\}}^{\oplus N}={\mathcal{K}}_j\oplus {\mathcal{K}}_j^{\perp}.$$
 Then the orthogonal projection $\pi_j$ from $\O^{\oplus N}_{B\setminus\{0\}}$ to the orthogonal complement of ${\mathcal{K}}_j$ converges to $\pi_\infty$ away from the center of $\Delta_i$. Now by assumption ${\mathcal{K}}_\infty$ is locally free over $\Delta_i$ so we may fix a smooth trivialization of ${\mathcal{K}}_\infty$ over $\Delta_i$. In particular this yields another trivialization of ${\mathcal{K}}_\infty$ over $\partial \Delta_i$. Fix an arbitrary smooth trivialization of ${\mathcal{K}}_\infty^{\perp}$ over $\partial \Delta_i$, so together we get a new trivialization of 
 $$\O_{B\setminus\{0\}}^{\oplus N}={\mathcal{K}}_\infty\oplus {\mathcal{K}}_\infty^{\perp}$$ over $\partial \Delta_i$. 
Using the smooth convergence above we may for $j\gg1$ find  smooth trivializations of ${\mathcal{K}}_j$ and ${\mathcal{K}}_j^{\perp}$ so that they are identified with ${\mathcal{K}}_\infty$ and ${\mathcal{K}}_\infty^{\perp}$ (as complex vector bundles). Notice by \cite{CS2}, Lemma 4.1, the relative Chern-Simons integral does not depend on the choice of the trivialization of the bundle, so using the new trivialization we may write
$$CS(G_\infty\oplus p_\infty^*A_\infty, G_j\oplus p_j^*A_j)=CS(G_\infty, G_j)+CS(p_\infty^*A_\infty, p_j^*A_j). $$
Now  since the above trivializations of ${\mathcal{K}}_\infty$ and ${\mathcal{K}}_j$ on $\partial \Delta_i$ extend over $\Delta_i$, it follows that  (see Lemma 4.1 in \cite{CS2})
$$CS(G_j, G_\infty)=\int_{\Delta_i} Tr(F_{G_j} \wedge F_{G_j})-Tr(F_{G_\infty} \wedge F_{G_\infty}).$$
So it suffices to prove 
$$\lim_{j\rightarrow\infty}\int_{\partial \Delta_i} CS(p_\infty^* A_\infty, p_j^* A_{j})=0.$$
Essentially, this follows from the fact that $p_\infty$ comes from the limit of maps which are initially globally defined over $\Delta_i$. By assumption, there exists a sequence of gauge transforms $g_j$  and Hermitian isomorphisms $P_j: \E_j \rightarrow \E_\infty$, both defined away from $\Gamma$, so that $(P_j^{-1})^*(g_j.A_{j})$ converges smoothly to $A_\infty $ away from $\Gamma$. Furthermore, $P_j g_j p_j$ converge to $p_\infty$ smoothly away from $\Gamma$. Given this, we have 
$$
\begin{aligned}
& \int_{\partial \Delta_i} CS(p_\infty^* A_\infty,p_j^*A_{j})\\
=&\lim_{j'\rightarrow\infty} \int_{\partial \Delta_i} CS((P_{j'}g_{j'} p_{j'})^* (P_{j'}^{-1})^*g_{j'} . A_{j'}, p_j^* A_j)\\
=&\lim_{j'\rightarrow\infty} \int_{\partial \Delta_i} CS((p_{j'})^*A_{j'}, p_j^* A_j)\\
=&\lim_{ j'\rightarrow\infty} \int_{\Delta_i} \Tr(F_{A_{j'}} \wedge F_{A_{j'}}) - \Tr(F_{A_j} \wedge F_{A_j}),
\end{aligned}
$$
where the last equality follows from the fact that $p_j$ and $p_j'$ are globally defined over $\Delta_i$ and Lemma 4.1 in \cite{CS2}. Now let $j$ tend to infinity,  the last term converges to zero by Lemma \ref{analyticmultiplictyformula}. 

 \end{proof}

\end{document}